\def\showFigure#1{#1}
\documentclass{amsart}

\usepackage[numbers, sort&compress]{natbib}
\usepackage{amsmath,amsthm,amsfonts,amssymb}
\usepackage{enumerate} 
\usepackage{mathabx}
\usepackage{color}
\usepackage{hyperref}
\usepackage{graphicx}

\newtheorem{theorem}{Theorem}[section]
\newtheorem{lemma}[theorem]{Lemma}
\newtheorem{proposition}[theorem]{Proposition}

\theoremstyle{definition}
\newtheorem{remark}[theorem]{Remark}
\newtheorem{algo}{Algorithm}

\numberwithin{equation}{section}

\newcommand{\comment}[1]{}
\def\fddto{\stackrel{f.d.d.}\weakto}
\newcommand{\ind}{{\bf 1}}

\def\inddd#1{{\ind}_{\left\{#1\right\}}}

\newcommand{\proba}{\mathbb P}
\renewcommand{\P}{\mathbb P}
\newcommand{\esp}{{\mathbb E}}

\newcommand{\inv}{^{-1}}
\newcommand{\cov}{{\rm{Cov}}}
\newcommand{\var}{{\rm{Var}}}

\newcommand{\eqnh}{\begin{eqnarray*}}
\newcommand{\eqne}{\end{eqnarray*}}
\newcommand{\eqnhn}{\begin{eqnarray}}
\newcommand{\eqnen}{\end{eqnarray}}
\newcommand{\equh}{\begin{equation}}
\newcommand{\eque}{\end{equation}}

\def\summ#1#2#3{\sum_{#1 = #2}^{#3}}

\def\sif#1#2{\sum_{#1=#2}^\infty}

\newcommand{\eqd}{\stackrel{d}{=}}

\def\topp#1{^{(#1)}}

\def\nn#1{{\left\|#1\right\|}}

\def\abs#1{\left|#1\right|}

\def\ccbb#1{\left\{#1\right\}}

\def\pp#1{\left(#1\right)}

\def\bb#1{\left[#1\right]}

\def\mmod{\ {\rm mod }\ }

\def\aa#1{\left\langle #1\right\rangle}

\def\vv#1{{\boldsymbol #1}}

\def\vvs{{\vv s}}
\def\vvt{{\vv t}}

\def\vvx{{\vv x}}

\def\mand{\mbox{ and }}
\def\qmand{\quad\mbox{ and }\quad}

\def\mwith{\mbox{ with }}
\def\qmwith{\quad\mbox{ with }\quad}
\def\mfa{\mbox{ for all }}

\def\mmas{\mbox{ as }}

\def\wt#1{\widetilde{#1}}
\def\wb#1{\overline{#1}}

\def\PPP{{\rm PPP}}

\def\weakto{\Rightarrow}

\def\R{{\mathbb R}}
\def\Rd{{\mathbb R^d}}
\def\N{{\mathbb N}}

\def\B{{\mathbb B}}

\def\E{{\mathbb E}}
\def\G{{\mathbb G}}

\def\S{{\mathbb S}}
\def\calA{\mathcal A}
\def\calB{\mathcal B}

\def\calE{\mathcal E}
\def\calF{\mathcal F}

\def\calN{\mathcal N}

\def\calS{\mathcal S}

\def\dbb#1{\left\ldbrack#1\right\rdbrack}
\def\dd{{\mathsf d}}

\title[Simulations for Karlin random fields]{Simulations for Karlin random fields}
\author{Zuopeng Fu}\address{Zuopeng Fu\\Department of Mathematical Sciences\\University of Cincinnati\\2815 Commons Way\\Cincinnati, OH, 45221-0025, USA.}\email{fuzg@mail.uc.edu}
\author{Yizao Wang}\address{Yizao Wang\\Department of Mathematical Sciences\\University of Cincinnati\\2815 Commons Way\\Cincinnati, OH, 45221-0025, USA.}\email{yizao.wang@uc.edu}

\begin{document}\sloppy
\begin{abstract}
We investigate the simulation methods for a large family of stable random fields that appeared in the recent literature, known as the Karlin stable set-indexed processes.  We exploit a new representation and implement the procedure introduced by \citet{asmussen01approximations} by first decomposing the random fields into large-jump and small-jump parts, and simulating each part separately.  
As special cases, simulations for several manifold-indexed processes are considered, and adjustments are introduced accordingly in order to improve the computational efficiency. 
\end{abstract}

\date{\today}
\maketitle
\section{Introduction}
This paper is a continuation of our earlier work on Karlin stable set-indexed processes in \citep{fu20stable}.
In the most general framework, a Karlin stable set-indexed process is associated to a measure space $(E,\calE,\mu)$ with a $\sigma$-finite measure $\mu$ and an index set $\calA\subset\calE$ such that for each $A\in\calA$, $\mu(A)<\infty$. Fix $(E,\calE,\mu)$ and $\calA$. Then, the  corresponding Karlin stable set-indexed process, denoted by $Y_{\alpha,\beta}$ for $\alpha\in(0,2]$ and $\beta\in(0,1)$, is defined via the following stochastic-integral representation \citep[Remark 3.2]{fu20stable} 
\equh\label{eq:xi_integral}
\ccbb{Y_{\alpha,\beta}(A)}_{A\in\calA}\eqd \ccbb{\int_{\R_+\times\Omega'}\inddd{[\calN'^{(r)}(\omega')](A) \ \rm odd}M_\alpha(drd\omega')}_{A\in\calA},
\eque
where $(\Omega',\calF',\proba')$ is another probability space, on which $\calN'^{(r)}$ is a Poisson point process on $(E,\calE)$ with intensity measure $r\mu$, $r>0$,  $M_\alpha$ is an S$\alpha$S random measure on $\R_+\times\Omega'$ with control measure $c_\beta r^{-\beta-1}drd\proba'$, and
\[
c_\beta := \frac{\beta 2^{1-\beta}}{\Gamma(1-\beta)}.
\]
We shall refer to a Karlin stable set-indexed process as a Karlin random field in short from time to time, and its law is throughout understood in their finite-dimensional distributions (so is the notation `$\eqd$').
The constant $c_\beta$ is chosen so that 
$\esp\exp\pp{i\theta Y_{\alpha,\beta}(A)} 
= \exp\pp{-\mu^\beta(A)|\theta|^\alpha}, \alpha\in(0,2)$. 
Recent developments on the Karlin random fields include \citep{durieu16infinite,durieu20infinite}, based on the original work of \citet{karlin67central}. The Karlin model \citep{karlin67central} is an infinite urn model that plays a fundamental role in combinatorial stochastic processes \citep{pitman06combinatorial,gnedin07notes}.

The abstract representation \eqref{eq:xi_integral} of Karlin random fields provides a stochastic-integral representation for set-indexed fractional Brownian motions ($\alpha=2$, see Lemma \ref{lem:sifBm} below) \citep{herbin06set} and hence extends set-indexed fractional Brownian motions to stable cases. It has a few notable manifold-indexed examples  as summarized below.
 When $\alpha=2$,  these are well-investigated centered Gaussian random fields, with the covariance functions recalled respectively. 
\begin{enumerate}[(i)]
\item Karlin stable processes, with 
\[
(E,\calE,\mu) = (\R_+,\calB(\R_+),{\rm Leb}), \mand \{A_t\}_{t\ge0} = \{[0,t]\}_{t\ge0}.
\]
When $\alpha=2$, these are fractional Brownian motions with Hurst index $\beta/2\in(0,1)$, with covariance function
\equh\label{eq:cov_fBm}
\frac12\pp{s^\beta + t^\beta - |s-t|^\beta}, s,t\ge 0.
\eque
\item
Multiparameter fractional stable fields, with
\[
(E,\calE,\mu) = (\R_+^2,\calB(\R_+^2),{\rm Leb}), \mand \ccbb{A_\vvt} _{\vvt\in\R_+^2} = \ccbb{[\vv0,\vvt]}_{\vvt\in\R_+^2}.
\]
When $\alpha=2$, these are multiparameter fractional Brownian motions introduced in \citep{herbin07multiparameter}, with covariance function \equh\label{eq:cov_mfBm}
\frac12\pp{{\rm Leb}([\vv0,\vvs])^\beta+{\rm Leb}([\vv0,\vvt])^\beta - {\rm Leb}([\vv0,\vvs]\Delta[\vv0,\vvt])^\beta}, \vvs,\vvt\ge 0.
\eque
We write $[\vv a,\vv b] = [a_1,b_1]\times[a_2,b_2]$ for $\vv a = (a_1,a_2),\vv b=(b_1,b_2)\in\R_+^2$.

\item Fractional L\'evy--Chentsov stable fields, with
\[
(E,\calE,\mu) = \pp{\S^1\times \R_+,\calB(\S^1\times\R_+),d\vv sdr},
\]
where $d\vv sdr$ is the product measure of the uniform measure $d\vv s$ on $\S^1$ and the Lebesgue measure $dr$ on $\R_+$, 
and\[
A_\vvt = \ccbb{(\vvs, r): \vvs \in \S^1, 0 < r < \aa{\vvs, \vvt}}, \vvt\in\R^2.
\]
This family and the one in the next examples extend the well-known L\'evy--Chentsov stable fields \citep{samorodnitsky94stable,takenaka10stable}. 
With $\alpha=2$, these are  the fractional L\'evy Brownian fields with Hurst index $\beta/2\in(0,1)$ \citep{samorodnitsky94stable}, with covariance function 
\equh\label{eq:cov_fLBf}
\frac12\pp{\nn{\vvs}^\beta_2+\nn\vvt^\beta_2 - \nn{\vvt-\vvs}_2^\beta}, \vvs,\vvt\in\R^2.
\eque
\item 
Spherical fractional L\'evy--Chenstov stable fields, with
\[
(E,\calE,\mu) = (\S^2,\calB(\S^2),d \vvs),
\]
where $d\vvs$ is the Lebesgue measure on the unit sphere $\S^2$ in $\R^3$, and
\[
A_{\vv x}=H_{\vv x} \triangle H_{\vv o}, \vv x \in \S^2 \qmwith H_{\vv x}: = \ccbb{\vv y \in \S^2: \langle \vv x, \vv y\rangle > 0},
\]
where $\vv o\in \S^2$ is an arbitrary fixed point. When $\alpha=2$, these are spherical fractional Brownian motions  with Hurst index $\beta/2\in(0,1)$ \citep{istas05spherical}, with covariance function ($\mathsf d_{\S^2}$ is the geodesic metric on $\S^2$)
\[
\frac12\pp{\mathsf d_{\S^2}^\beta(\vv o,\vvx)+\mathsf d_{\S^2}^\beta(\vv o,\vv y) - \mathsf d_{\S^2}^\beta(\vvx,\vv y)}, \vvx,\vv y\in\S^2.
\]
\end{enumerate}

In this paper we investigate the corresponding simulation methods. 
Simulation methods for Gaussian random fields have been extensively studied in theory and broadly applied in various fields (see e.g.~\citep{bierme19introduction,cohen13fractional,kroese15spatial} for overviews, and \citep{gelbaum14simulation,vanwyk15power} for some recent attempts for models with more general manifold index sets). As for stable processes and more generally infinitely-divisible processes, the foundation of simulation methods has been laid down in the seminal work of \citet{asmussen01approximations}. They focused on L\'evy processes in the original paper, but essentially the same idea applies to more general stable processes and infinitely-divisible processes, carried out in details by  Lacaux and coauthors later \citep{lacaux04series,lacaux04real,cohen08general}. These references served as our starting point. 
Namely, it has been well understood since then that in order to simulate an infinitely-divisible process, in practice one should first decompose the process into two independent components consisting of large and small jumps respectively, and then simulate each part separately. We shall follow the same idea here for the Karlin random fields (see Remark \ref{rem:Lacaux} for subtile differences between our framework and aforementioned ones), and the two parts are referred to as the large-jump and small-jump parts, respectively. 

The main contribution of this paper is two-folded. 
\begin{enumerate}[(a)]
\item First, we develop a new representation for Karlin random fields, {\em when restricted to a bounded domain}: that is, the index set $\calA_0$ is such that there exists $E_0\in\calE$ with $\mu(E_0)<\infty$ and for all $A\in\calA_0$, $A\subset E_0$ (Theorem \ref{thm:1}). All the examples mentioned above, when simulated over a bounded domain, can be reduced to such a situation and hence the new representation applies.
The advantage of this new representation is that it provides a compound--Poisson representation for the large-jump part in the Asmussen--Rosi\'nski approach, and hence yields immediately an exact and straightforward simulation method for this part. This is in contrast to the developments in \citep{lacaux04series,lacaux04real,cohen08general}, where for most interesting examples the simulations for the large-jump part require approximation methods. 

\item We then apply the new representation to the aforementioned examples, and  propose adjustments accordingly in order to improve computational efficiency. Most notably, a straightforward implementation of the Asmussen--Rosi\'nski approach would meet computational issues even in the simplest case of $\R_+$-indexed Karlin stable processes. The issues are due to the fact that the new representation is essentially based on the so-called {\em odd-occupancy vector}, the law of which is determined by the $\beta$-Sibuya distribution (of which the tail is regularly varying with index $-\beta$, $\beta\in(0,1)$). Sampling directly from the heavy-tailed Sibuya distribution is very inefficient in practice, and in a couple situations we managed to propose a computational efficient method to sample the odd-occupancy vector directly without sampling the Sibuya distribution. 
\end{enumerate}
\begin{figure}[ht!]
\showFigure{
\begin{center}
\includegraphics[width=.9\textwidth]{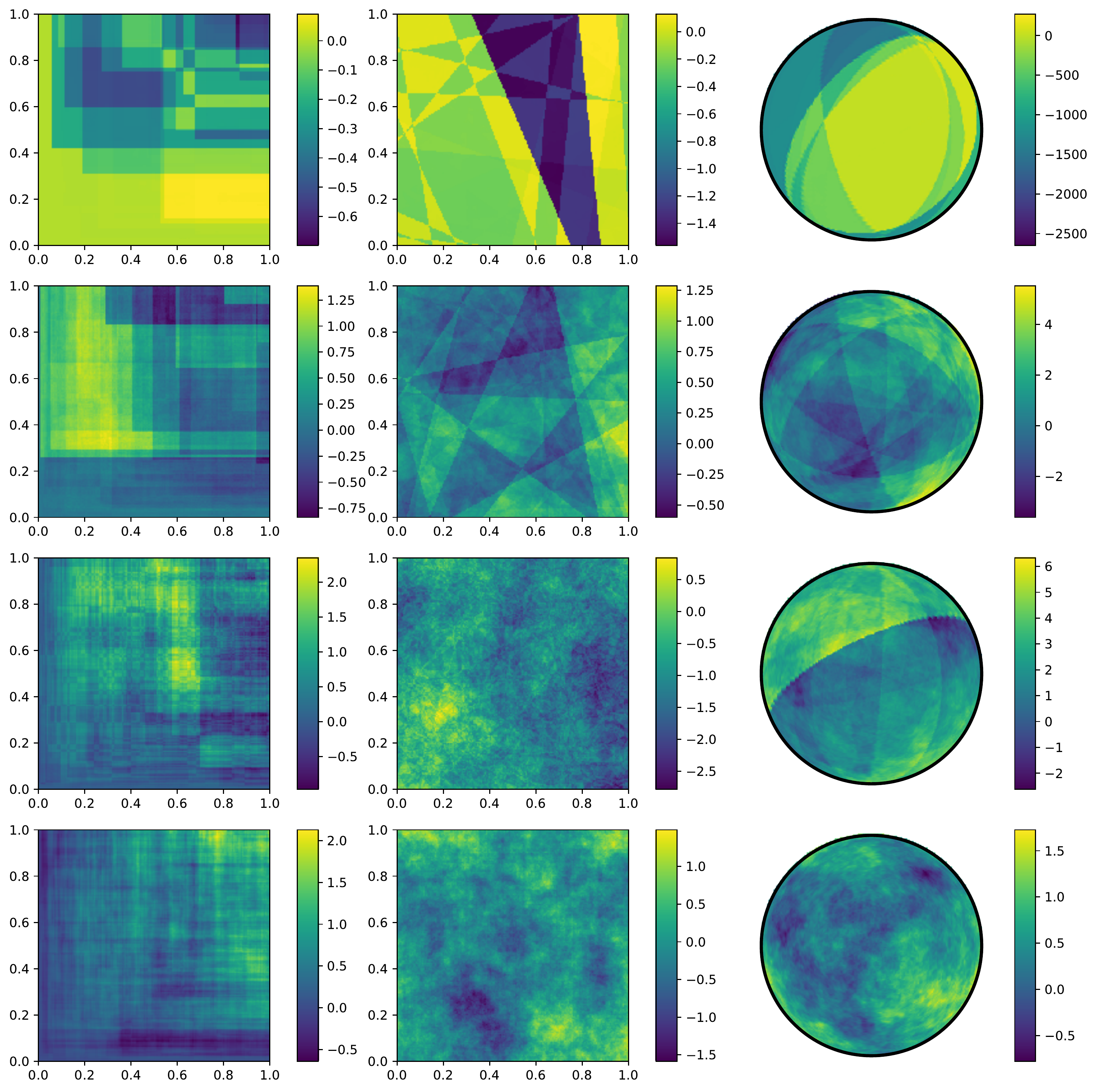}
\end{center}
}
 \caption{Simulations for  
$\R_+^2$-indexed multiparameter fractional stable fields (left),
 $\R^2$-indexed fractional L\'evy--Chentsov stable fields (middle) and $\mathbb S^2$-indexed fractional L\'evy--Chentsov stable fields (right), with $\alpha = 0.5$ (top) $\alpha = 1.2$ (second row), $\alpha = 1.8$ (third row) and $\alpha = 2$ (bottom, Gaussian), and all with $\beta = 0.8$. The Gaussian cases correspond to multiparameter fractional Brownian motions, fractional L\'evy Brownian fields, and spherical fractional Brownian motions, respectively. 
 }
 \label{fig:samples} 
\end{figure}

In Figure \ref{fig:samples} we provide a few simulation examples of the processes of our interest. Note that when $\alpha<2$ these are only approximated samplings. 
Curiously, for fractional L\'evy--Chentsov stable fields, the odd-occupancy vectors are functionals of models from stochastic geometry \citep{lantuejoul02geostatistical,schneider08stochastic}, as illustrated in Figures \ref{fig:odd_fLC} and \ref{fig:odd_sfLC} later. So fractional L\'evy--Chentsov stable fields can be thought of aggregations of models from stochastic geometry.

The paper is organized as follows. Section \ref{sec:new} introduces a new  representation for the Karlin random fields, and explains the general strategy for simulations. Section \ref{sec:examples} investigates a few examples and explains how improvement can be made regarding efficiency of the simulations. Appendix  \ref{sec:general} provides a review on the general framework of \citet{asmussen01approximations} applied to stable processes. 

\section{Karlin stable set-indexed processes}\label{sec:new}
\subsection{A new representation}
We develop a new representation of Karlin stable set-indexed processes, when restricted to a bounded domain. More precisely, fix some $E_0\in\calE$ with $\mu(E_0)<\infty$ and consider an index set $\calA_0$ such that $A\subset E_0$ for all $A\in\calA_0$.  
We let $Q_\beta$ be a random variable with the Sibuya distribution with parameter $\beta\in(0,1)$, determined by $\esp z^{Q_\beta} = 1-(1-z)^\beta$ for all $z\in[0,1]$ \citep{sibuya79generalized}. Equivalently, $Q_\beta$ takes values from $\N$ with
\[
\P (Q_{\beta} = k) = \frac\beta{\Gamma(1-\beta)}\frac{\Gamma(k-\beta)}{\Gamma(k+1)}
\sim \frac\beta{\Gamma(1-\beta)}k^{-1-\beta}
\mmas k\to\infty,
\]
so it is a heavy-tailed distribution without finite $\beta$-th moment. Throughout, the following random closed set $R_\beta$ in $\calF_0(E_0)$, the space of non-empty closed subsets of $E_0$ (see \citep{molchanov17theory} for background on random closed sets), plays a fundamental role for the Karlin random fields
\equh\label{eq:R_beta}
R_\beta:=\bigcup_{i=1}^{Q_\beta}\{U_i\},
\eque
where $\{U_i\}_{i\in\N}$ are i.i.d.~random elements from $E_0$ with the law $\mu_{E_0}(\cdot):=\mu(\cdot\cap E_0)/\mu(E_0)$
 independent from $Q_\beta$ introduced before. So $R_\beta$ is a random closed set taking values in $\calF_0(E_0)$. 
The new representation is summarized as follows.

\begin{theorem} \label{thm:1}
Assume $E_0$ and $\calA_0$ as above. For all $\alpha\in(0,2]$, $\beta\in(0,1)$, 
the Karlin set-indexed stable process \eqref{eq:xi_integral} restricted to $\calA_0$ has the stochastic-integral representation
\equh\label{eq:integral}
\ccbb{Y_{\alpha,\beta}(A)}_{A\in\calA_0}\eqd\ccbb{\int_{\Omega'}\inddd{|R_\beta'(\omega')\cap A|\ \rm odd}\wt M_\alpha(d\omega')}_{A\in\calA_0},
\eque
where $(\Omega',\calF',\proba')$ is another probability space, on which $R_\beta'(\omega)$ is a random element in $E_0$ with the same law as $R_\beta$, and $\wt M_\alpha$ is an S$\alpha$S random measure on $\Omega'$ with control measure $2^{1-\beta}\mu^\beta(E_0)\cdot\proba'$. 
\end{theorem}
\begin{proof}
We compute the characteristic function of finite-dimensional distributions. For $d\in\N$,  $A_1,\dots,A_d\in\calA_0$ and $\theta_1,\dots,\theta_d\in\R$, we have
\[
\esp \exp \pp{i\summ j1d \theta_jY_{\alpha,\beta}(A_j)} = \exp\pp{-\int_{\R_+\times\Omega'}\abs{\summ j1d \theta_j\inddd{\calN'^{(r)}(A_j)\ \rm odd}}^\alpha c_\beta r^{-\beta-1}d\proba'}.
\]
Note that by the property of Poisson point processes, there exists a measure $\wt\nu$ on $\N$ such that the above is the same as, with $\{U'_i\}_{i\in\N}$ as i.i.d.~random variables with law $\mu_{E_0}$ (defined on some probability space denoted by $(\Omega',\calF',\proba')$ without loss of generality),
\equh\label{eq:1}
\exp\pp{-\sif k1 \int_{\Omega'}\abs{\summ j1d \theta_j\inddd{\abs{\bigcup_{i=1}^k\{U'_i\}\cap A_j}\ \rm odd}}^\alpha\wt\nu(\{k\})d\proba'}.
\eque
The values of $\wt\nu$ can be computed as 
\begin{align*}
\wt\nu(\{k\})& = c_\beta\int_0^\infty r^{-\beta-1}\proba\pp{\calN\topp r(E_0) =k}dr 
 = c_\beta \int_0^\infty r^{-\beta-1}\frac{(r\mu(E_0))^k}{k!}e^{-r\mu(E_0)}dr\\
&=\frac{\beta2^{1-\beta}}{\Gamma(1-\beta)} \cdot\mu^\beta(E_0)\cdot \frac{\Gamma(k-\beta)}{\Gamma(k+1)} = 2^{1-\beta}\mu^\beta(E_0)\proba(Q_\beta = k), \mfa k\in\N.
\end{align*}
Then, \eqref{eq:1} becomes, letting $Q_\beta'$ be a $\beta$-Sibuya random variable on $(\Omega',\calF',\proba')$, independent from $\{U'_i\}_{i\in\N}$,
\begin{multline*}
\exp\pp{-2^{1-\beta}\mu^\beta(E_0)\int_{\Omega'}\abs{\summ j1d \theta_j\inddd{\abs{\bigcup_{i=1}^{Q_\beta'}\{U_i'\}\cap A_j}\ \rm odd}}^\alpha d\proba'}\\
= \exp\pp{-2^{1-\beta}\mu^\beta(E_0)\int_{\Omega'}\abs{\summ j1d \theta_j\inddd{\abs{R_\beta'\cap A_j}\ \rm odd}}^\alpha d\proba'}.
\end{multline*}
This completes the proof.
\end{proof}
\begin{remark}
Let $P_{\calN\topp r,+}$ denote the induced probability measure of $\calN\topp r$ (as a random closed set) restricted to $\calF_0(E_0)$; in particular $P_{\calN\topp r,+}$ is a sub-probability measure for all $r>0$ (i.e.~$P_{\calN\topp r,+}(\calF_0(E_0))<1$).
Let $P_{R_\beta}$ denote the induced probability measure on $\calF_0(E_0)$ by $R_\beta$. We have essentially proved 
\equh\label{eq:law}
P_{R_\beta}(\cdot) = \frac{\beta}{\Gamma(1-\beta)\mu^\beta(E_0)}\int_0^\infty r^{-\beta-1}P_{\calN\topp r,+}(\cdot)dr
\eque
as a probability measure on $\calF_0(E_0)$. 
The right-hand side, in the language of Radon point measures instead of random closed sets,  appeared in \citep[Eq.(3.1)]{fu20stable} as $\mu_\beta(\cdot)/(2^{1-\beta}\mu^\beta(E_0))$ and played a central role in the representations of Karlin random fields therein.
\end{remark}
The integral representations \eqref{eq:integral} with  $\alpha = 2$ corrponds to {\em set-indexed fractional Brownian motions} with Hurst index $H = \beta/2\in(0,1/2)$ \citep{herbin06set}. These are centered Gaussian processes, denoted by $\{\B_\mu^{\beta/2}(A)\}_{A\in\calA_0}$, with
\equh\label{eq:set_indexed}
\cov \pp{\B_\mu^{\beta/2}(A_1), \B_\mu^{\beta/2}(A_2)} =\frac{1}{2} \pp{\mu^{\beta}(A_1) + \mu^{\beta}(A_2) - \mu^{\beta}(A_1 \Delta A_2)}, A_1,A_2\in\calA_0.
\eque
\begin{lemma}\label{lem:sifBm}
Let $Y_{2,\beta}$ be as in \eqref{eq:integral}. Then,
\[
\ccbb{Y_{2,\beta}(A)}_{A\in\calA_0}\eqd \ccbb{\B_\mu^{\beta/2}(A)}_{A\in\calA_0}.
\] 
\end{lemma}
A stronger result, including a decomposition of set-indexed fractional Brownian motions, was already proved in \citep[Section 3.3]{fu20stable}. We include a quick proof for a weaker result here, and we shall need the computation \eqref{eq:law1} below later.  
\begin{remark}
Note that our covariance formula differs from the one in \citep[Section 3.3]{fu20stable} by a factor of 2. This is because therein for a streamlined presentation we took the convention that the characteristic function for a stochastic integral is $\esp \exp(i\theta\int_S fdM_\alpha) = \exp\pp{-|\theta|^\alpha\int_S|f|^\alpha d\mu}$ {\em for all $\alpha\in(0,2]$}. With $\alpha=2$ this  is different from the common convention (considered above) under which the characteristic function is $\exp\pp{-(1/2)|\theta|^2\int_S|f|^2d\mu}$ instead (e.g.~\citep[Remark 2.9]{fu20stable}).
\end{remark}

\begin{proof}[Proof of Lemma \ref{lem:sifBm}]
We compute
\begin{align*}
\cov\pp{Y_{2,\beta}(A_1),Y_{2,\beta}(A_2)} = 2^{1-\beta}\mu^\beta(E_0)\cdot \proba\pp{R_\beta(A_1)\ {\rm odd}, R_\beta(A_2)\ {\rm odd}}.
\end{align*}
We shall use the identity \eqref{eq:law} instead of using the representation \eqref{eq:R_beta} involving $Q_\beta$. Namely, 
\begin{multline}\label{eq:law1}
\proba\pp{|R_\beta \cap A_1|\ {\rm odd}, |R_\beta \cap A_2|\ {\rm odd}} \\=\frac{\beta}{\Gamma(1-\beta)\mu^\beta(E_0)}\int_0^\infty r^{-\beta-1}\proba\pp{\calN\topp r(A_1)\ {\rm odd}, \calN\topp r(A_2)\ {\rm odd}}dr.
\end{multline}
We first compute the probability in the integrand. By discussing the even/odd cardinalities of $A_1\setminus A_2, A_2\setminus A_1, A_1\cap A_2$, we see that it is the same as
\begin{multline*}
\proba\pp{\calN\topp r(A_1)\ {\rm odd}, \calN\topp r(A_2)\ {\rm odd}} \\ =
\frac12\bb{\proba\pp{\calN\topp r(A_1) \ {\rm odd}} + \proba\pp{\calN\topp r(A_2)\ {\rm odd}} - \proba\pp{\calN\topp r(A_1\Delta A_2)\ {\rm odd}}}.
\end{multline*}
So \eqref{eq:law1} becomes
\[
\frac{\beta}{\Gamma(1-\beta)\mu^\beta(E_0)}\int_0^\infty \frac{r^{-\beta-1}}4\pp{1-e^{-2\mu(A_1)r}+ 1-e^{-2\mu(A_2)r}-1+e^{-2\mu(A_1\Delta A_2)r}}dr.
\]
With $\int_0^\infty \beta r^{-\beta-1}(1-e^{-ar})dr = a^\beta \Gamma(1-\beta)$ for $a>0$, the above becomes then
\begin{multline}\label{eq:A1A2}
\proba\pp{|R_\beta \cap A_1|\ {\rm odd}, |R_\beta \cap A_2|\ {\rm odd}} \\
= \frac1{2^{1-\beta}\mu^\beta(E_0)}\cdot \frac{1}{2} \pp{\mu^{\beta}(A_1) + \mu^{\beta}(A_2) - \mu^{\beta}(A_1 \Delta A_2)}.
\end{multline}
We now see that $Y_{2,\beta}$ and $\B^{\beta/2}_\mu$ share the same covariance function. This completes the proof.
\end{proof}

When restricted to $\alpha\in(0,2)$ the Karlin random field with representation \eqref{eq:integral} has the following series representation (see \citep[Theorem 3.10.1]{samorodnitsky94stable}), 
\equh\label{SetKarlin2}
\ccbb{Y_{\alpha,\beta}(A)}_{A \in \calA_0} \eqd \ccbb{\sum_{j \in \N} \eta_{\alpha,j}\inddd{\abs{R_{\beta, j} \cap A}\ \rm odd}}_{A \in \calA_0},
\eque
where $\{\eta_{\alpha,j}\}_{j\in\N}$ are enumerations of a Poisson point process on $\wb\R\setminus\{0\}$ with intensity 
\[
2^{1-\beta}\mu^\beta(E_0)\cdot\frac{\alpha C_\alpha}2|x|^{-\alpha-1}, x\ne0,
\]
with
\[
C_\alpha := \pp{\int_0^\infty x^{-\alpha-1}\sin xdx}\inv, \alpha\in(0,2),
\]
and $\{R_{\beta,j}\}_{j\in\N}$ are i.i.d.~copies of $R_\beta$, independent from $\{\eta_{\alpha,j}\}_{j\in\N}$.

\subsection{A general simulation framework}
The framework of \citet{asmussen01approximations} applies to $\{Y_{\alpha,\beta}(A)\}_{A\in\calA_0}$ as follows. Take the random series on the right-hand side of \eqref{SetKarlin2} as the definition of $Y_{\alpha,\beta}(A)$. Then given $\epsilon>0$, we write
\[
Y_{\alpha,\beta} (A) = Y_{\alpha,\beta}^{\epsilon, 1} (A) + Y_{\alpha,\beta}^{\epsilon, 2} (A)
\]
as the sum of the large-jump and the small-jump parts of the original process given by
\begin{align*}
Y_{\alpha,\beta}^{\epsilon,1}(A)& :=\sum_{j \in \N} \eta_{\alpha,j}\inddd{\abs{R_{\beta, j} \cap A}\ \rm odd}\inddd{\eta_{\alpha,j}>\epsilon},\\
Y_{\alpha,\beta}^{\epsilon,2}(A)& :=\sum_{j \in \N} \eta_{\alpha,j}\inddd{\abs{R_{\beta, j} \cap A}\ \rm odd}\inddd{\eta_{\alpha,j}\le \epsilon}, A\in\calA_0,
\end{align*}
respectively.
The large-jump part has a compound-Poisson representation
\equh \label{eq:large-jump}
\ccbb{Y_{\alpha,\beta}^{\epsilon, 1} (A)}_{A\in\calA_0} \eqd \ccbb{\sum_{j =1}^{N_{\alpha,\epsilon}}  V_{\alpha,\epsilon,j} D_{j,A}}_{A\in\calA_0} \mwith D_{j,A}: = \inddd{\abs{R_{\beta,j} \cap A}\ \rm odd}, A\in \calA_0,
\eque
where $N_{\alpha,\epsilon}$ is a Poisson random variable with parameter $2^{1-\beta}\mu^\beta(E_0)C_\alpha\epsilon^{-\alpha}$ and $\{V_{\alpha,\epsilon,j}\}_{j \in \N}$ are i.i.d.~symmetric random variables with probability density function $\epsilon^\alpha (\alpha/2)|y|^{-\alpha-1}\inddd{|y|>\epsilon}$, $\{R_{\beta,j}\}_{j\in\N}$ are i.i.d.~copies of $R_\beta$, and all random variables are independent. 

For the small-jump part,
one can show the following.
\begin{proposition}\label{prop:small1}With the notations above,
\[
\ccbb{\frac{Y_{\alpha,\beta}^{\epsilon, 2} (A)}{\sigma_\alpha(\epsilon)}}_{A \in \calA_0} \stackrel{f.d.d.}\weakto \ccbb{\B_\mu^{\beta/2}(A)}_{A \in \calA_0},
\]
as $\epsilon\downarrow0$,
where $\{\B_\mu^{\beta/2}(A)\}_{A \in \calA_0}$ is the set-indexed fractional Brownian motion \citep{herbin06set} with the covariance function  \eqref{eq:set_indexed}.
\end{proposition}
\begin{proof}
The result follows from Proposition \ref{prop:small}  and Lemma \ref{lem:sifBm}. 
\end{proof}

Now we look into implementation issues. For our examples,  we always identify a set of indices $T$ (a subset of $\R^d$ or $\S^d$) to $\{A_t\}_{t \in T}\subset\calA_0$, and write simply from now on
\[
\ccbb{Y_{\alpha,\beta}(t)}_{t \in T} \equiv \ccbb{Y_{\alpha, \beta} (A_t)}_{t \in T},
\]
and similarly for the large-jump and small-jump parts. 
Now the above discussions suggest that the approximated process (in distribution) in simulation is
\equh\label{eq:approximation}
Y_{\alpha,\beta}(t) \approx Y_{\alpha,\beta}^{\epsilon,1}(t) + \sigma_\alpha(\epsilon)\B_\mu^{\beta/2}(t), t\in T.
\eque
While the large-jump part is compound Poisson and the approximated small-jump part is Gaussian, and both classes of stochastic processes in principle have exact simulation methods, computational issues arise quickly if one examines more closely.

For the large-jump part, clearly it suffices to sample the {\em odd-occupancy vector}
\[
\vv D = (D_{t_1},\dots,D_{t_n}) \qmwith D_t:= \inddd{|R_\beta\cap A_t|\ \rm odd},
\]
with a finite index lattice $T = \{t_1,\dots,t_n\}$ in practice.
A straightforward algorithm is the following.
\begin{algo}\label{algo:3}~
\begin{enumerate}
\item Generate a $\beta$-Sibuya random variable $Q_{\beta}$.
\item Sample $R_{\beta} \eqd \bigcup_{i=1}^{Q_{\beta}}\{U_i\}$.
\item Compute $\{D_{t}\}_{t\in T}$ based on the sampling of $R_{\beta}$. 
\end{enumerate}
\end{algo}
In order to sample $Q_\beta$ here,
we recall a nice expression due to \citet{sibuya79generalized}.  Namely, with $G_1$, $G_{\beta}$ and $G_{1-\beta}$ being three independent standard Gamma random variables with parameters $1$, $\beta$ and $1-\beta$, respectively, we have
\equh \label{sibuya1}
Q_{\beta} \eqd 1 + \mbox{Poisson}\pp{\frac{G_1G_{1-\beta}}{G_{\beta}}},
\eque
where the second term on the right-hand side is understood as a Poisson random variable with a random parameter. So in practice we could first sample the random parameter $\Lambda = G_1G_{1-\beta}/G_\beta$ and then a Poisson random variable with parameter $\Lambda$, and  add one to the sampled value at the end.

However, one should realize quickly that this algorithm is not computationally efficient, as the $\beta$-Sibuya distribution  does not have finite $\beta$-th moment \citep{pitman06combinatorial}. This could become quite cumbersome in practice as from time to time $Q_\beta$ may be hundreds of thousands, while the resolution $n$ in $T_n$ is at most a few hundreds. It turns out that for Karlin stable processes and multiparameter fractional stable processes, one can exploit further the structure of $\calA_0$ and sample $\vv D$ directly and much more efficiently, without sampling $Q_\beta$. 
\begin{remark}
In practice one should decide also what value of $\epsilon$ makes a good approximation in \eqref{eq:approximation}. One may choose the value according to the Berry--Esseen bound on the Gaussian approximation (see Remark \ref{rem:BE}), which for the marginal distribution in this case becomes (taking $(S,m) = (\Omega',2^{1-\beta}\mu^\beta(E_0)\cdot   \proba')$ and $f_t(\omega') = D_t(\omega') = \inddd{|R_\beta'(\omega')\cap A_t|\ {\rm odd}}$ as such that with respect to $\proba'$ $D_t'$ is a copy of $D_t$ before)
\[
C_{\rm BE} \frac1{(2^{1-\beta}\mu^\beta(E_0)\esp D_t)^{1/2}}\frac{(2-\alpha)^{3/2}}{(3-\alpha)\sqrt{\alpha C_\alpha}}\epsilon^{\alpha/2}  = C_{\rm BE} \frac1{\mu^{\beta/2}(A_t)}\frac{(2-\alpha)^{3/2}}{(3-\alpha)\sqrt{\alpha C_\alpha}}\epsilon^{\alpha/2},
\]
where we used 
\equh\label{eq:ED_t}
\esp D_t = \proba(R_\beta\cap A_t \ {\rm odd}) = \esp \pp{\frac12\bb{1-\pp{1-2\frac{\mu(A_t)}{\mu(E_0)}}^{Q_\beta}}} = 2^{\beta-1}\frac{\mu^\beta(A_t)}{\mu^\beta(E_0)}.
\eque 
In Figure \ref{fig:epsilon}, the values of $\epsilon = \epsilon_\alpha$ such that
\[
\frac{(2-\alpha)^{3/2}}{(3-\alpha)\sqrt{\alpha C_\alpha}}\epsilon^{\alpha/2} = 0.01
\]
is plotted, along with $C_\alpha$, $\sigma_\epsilon(\alpha)$ and $n_{\alpha,\epsilon} := C_\alpha \epsilon^{-\alpha}$, for $\alpha\in(0,2)$.  Note that $n_{\alpha,\epsilon} = \esp N_{\alpha,\epsilon}/(2^{1-\beta}\mu^\beta(E_0))$  and tells roughly (the terms depending on $\beta$ is dropped for simple comparison) how many independent copies are needed for the large-jump part \eqref{eq:large-jump}.  

From the plot we see that, first, the small-jump part is far from negligible for $\alpha$ close to 2. Second, for $\alpha<1$ the gain of approximating small-jump part is very limited, while the cost of simulating the large-jump part is huge. This is not surprising as it is well known that when $\alpha<1$ the series representation is absolutely summable, and the magnitudes of small jumps decay as $O(j^{-1/\alpha})$.
Therefore, in practice we choose not to apply the small-jump approximation for $\alpha<1$. See examples in Figure \ref{fig:samples} for $\alpha = 0.5$, where we set $\epsilon = 10^{-4}$.  
\end{remark}
\begin{center}
\begin{figure}[ht!]
\showFigure{
\includegraphics[width = .5\textwidth]{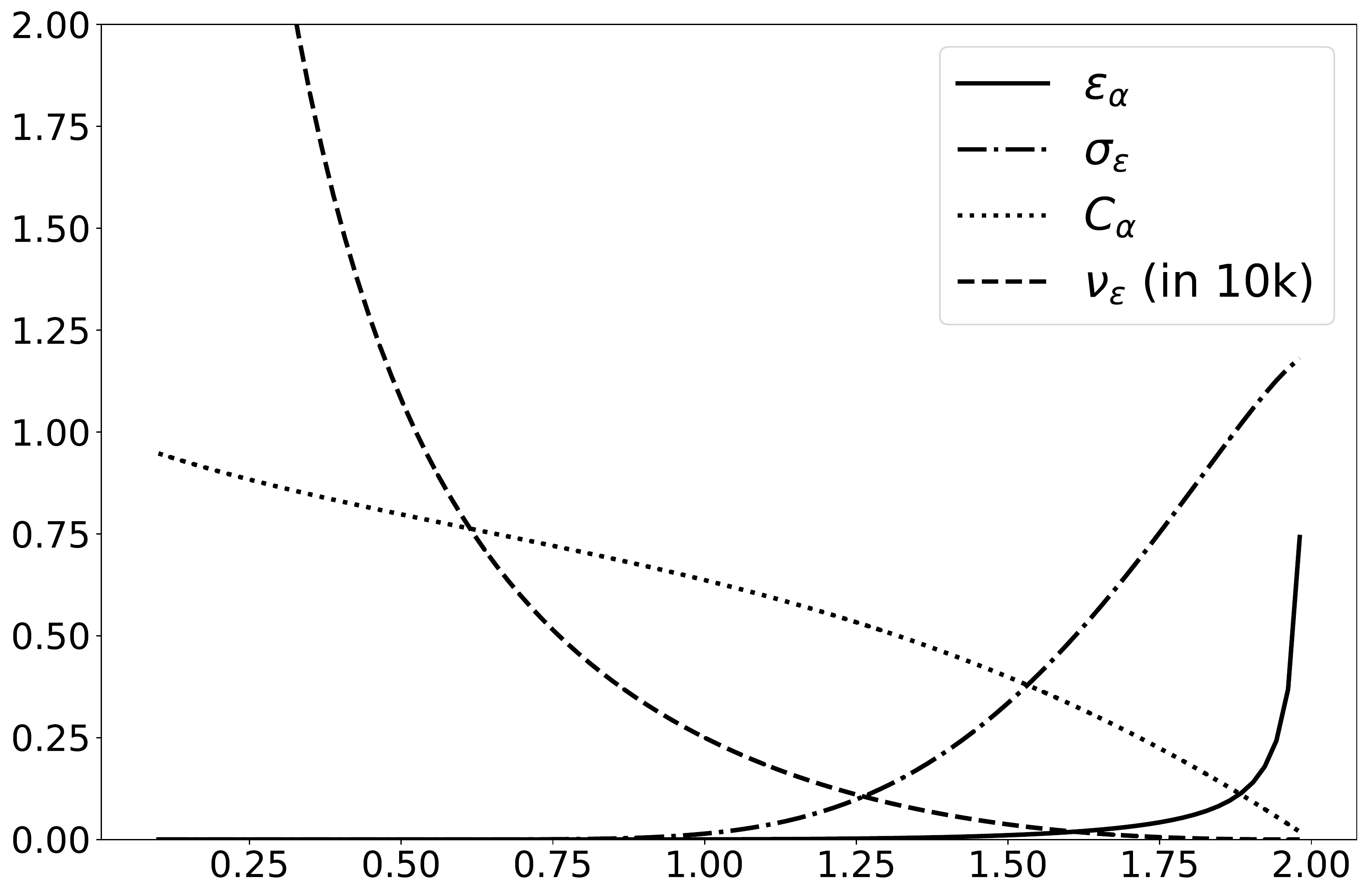}
}
\caption{Comparison of parameters.}\label{fig:epsilon}
\end{figure}
\end{center}
\begin{remark} 
Another numerical issue that we encountered in implementing Algorithm \ref{algo:3} is that, due to the fact that $\Lambda = G_1G_{1-\beta}/G_\beta$ is heavy-tailed, occasionally sampling $\Lambda$ returns a very huge number that forbids the computation to continue (e.g.~in Python on a 64-bit platform, an integer value is no bigger than $2^{63}-1$; the parameter of Poisson of $\Lambda$ can easily go beyond this order during say 1000 i.i.d.~sampling when $\beta<0.2$).  One way to go around this issue is to set up a threshold, say $\lambda_0$, and use ${\rm Poisson}(\Lambda\wedge \lambda_0)$ instead of ${\rm Poisson}(\Lambda)$ in Algorithm \ref{algo:3}.  Then, the probability that the threshold is exceeded at least once (and hence the simulation is only an approximation) is bounded by $\proba(\bigcup_{i=1}^{N_{\alpha,\epsilon}}\{\Lambda_i>\lambda_0\}) \le \esp N_{\alpha,\epsilon}\proba(\Lambda>\lambda_0)$. 
\end{remark}

For the small-jump part, the by-default method of applying the Cholesky decomposition to a covariance matrix of size $n\times n$ is computationally infeasible for high dimensions (with complexity $O(n^3)$, and $\R^2$- or $\S^2$-indexed processes a reasonable resolution requires $n$ to be at least $200^2$). In a few cases, we are in a fortunate situation that the set-indexed fractional Brownian motion is known to have  a fast and exact simulation method. The only exception is the case when it is a multiparameter fractional Brownian motion, for which we develop a fast approximation method. The simulation methods are summarized  in Table \ref{tab:1}.

In the next section we provide details for simulations for a few examples. Table \ref{tab:1} is a summary on where improvement can be made regarding simulation efficiency.
\begin{table}[ht!]
\caption{Summary of simulation methods for examples in Section \ref{sec:examples}. The column `$E$' indicates the underlying space $(E,\calE)$. The column `$\vv D$' indicates whether the odd-occupancy vector can be sampled in an efficient  way without sampling the entire $R_\beta$. The last column indicates the set-indexed fractional Brownian motion that approximates the small-jump part, and the corresponding simulation method. Acronyms used below are, fLCsf: fractional L\'evy--Chenstov stable field; mfsf: multiparameter fractional stable field; 
(m/s)fBm:  (multiparameter/spherical) fractional Brownian motion, fLBf: fractional L\'evy--Brownian field, CEM: circulant embedding method; IEM: intrinsic embedding method.}\label{tab:1}

\begin{tabular}{ |c|c|c|c|c| } 
\hline
Sec. & Example & $E$ &  $\vv D$ & set-indexed fBm \\
\hline
\ref{sec:KSP} & Karlin ($\R_+$-indexed fLCsf)   & $\R_{+}$ & fast & fBm, CEM \citep{wood94simulation,dietrich97fast}\\ 
\ref{sec:mfsf} & mfsf & $\R_{+}^2$ & fast  & mfBm,  Prop.~\ref{prop:small2} \\ 
\ref{sec:fLC} & $\R^2$-indexed fLCsf & $\S\times\R_+$ & slow & fLB, IEM \citep{stein02fast}\\
\ref{sec:sfLC} & $\S^2$-indexed fLCsf & $\S^2$ & slow & sfBm, CEM \citep{cuevas20fast} \\ 
\hline 
\end{tabular}
\end{table}
\section{Examples}\label{sec:examples}
Recall that we work with Karlin random fields $\{Y_{\alpha,\beta}(A_t)\}_{t\in T}$ in \eqref{SetKarlin2}, with a measure space $(E,\calE,\mu)$, $E_0\in\calE$ with $\mu(E_0)<\infty$, and an index set $\{A_t\}_{t\in T}$ such that $A_t \subset E_0$. The four examples summarized in Table \ref{tab:1} are worked out below one by one. 

\subsection{Karlin stable processes}\label{sec:KSP}
This example corresponds to the choice of 
\[
(E,\calE,\mu) \equiv (\R_+,\calB(\R_+),{\rm Leb}), E_0 = [0,1],\mand \{A_t\}_{t\in [0,1]} = \{[0,t]\}_{t\in[0,1]}.
\]

\subsubsection*{The large-jump part} In this case, we introduce an algorithm that improves significantly the efficiency of Algorithm \ref{algo:3} when simulating the odd-occupancy vectors, thanks to the structure of $\{A_t\}_{t\in[0,1]}$. Note that in simulation we only need to work with an index set $T = \{t_1,\dots,t_n\}$ with $0\le t_1<\cdots<t_n\le 1$. 
Let $N_{\Lambda_\beta}$ be a Poisson random variable with a random parameter $\Lambda_\beta: = G_1 G_{1-\beta}/G_{\beta}$, where $G_1$, $G_{\beta}$ and $G_{1-\beta}$ are as in \eqref{sibuya1}. We introduce this time
\[
\wt{R}_{\beta}: = \bigcup_{i=1}^{N_{\Lambda_\beta}} \ccbb{U_i}, 
\]
where $\{U_i\}_{i \in \N}$ are i.i.d.~uniform random variables over $(0,1)$ independent from $N_{\Lambda_\beta}$. Let $U$ be another uniform random variable independent from $\{U_i\}_{i\in\N}$. Define
\equh \label{eq:Mi}
M_i := \sum_{j=1}^{i}B_j + \inddd{U \in (0, t_i]} \qmwith B_i: = \inddd{\abs{\wt{R}_{\beta} \cap (t_{i-1}, t_{i}]}\ \rm odd}, i = 1, \dots, n.
\eque
Then, the Sibuya identity \eqref{sibuya1} says that $N_{\Lambda_\beta} +1\eqd Q_\beta$, and hence
\equh\label{eq:DM}
\ccbb{D_{t_i}}_{i=1, \dots, n} \eqd \ccbb{M_i \mmod 2}_{i=1, \dots, n}.
\eque
The advantage of this representation is that the random vector $\vv M = (M_1,\dots,M_n)$, or essentially $\vv B = (B_1,\dots,B_n)$, can be simulated as a collection of conditionally independent Bernoulli random variables, and hence with linear complexity in $n$ {\em without sampling the heavy-tailed $N_{\Lambda_\beta}$} (see Remark \ref{rem:complexity} below),  thanks to the following simple fact. 

\begin{lemma} \label{lem:B}
With the notations above, given $\Lambda_\beta$,  $\{B_i\}_{i=1, \dots, n}$ are conditionally independent Bernoulli random variables with parameters 
\equh\label{eq:p_i}
p_i(\Lambda_\beta) = \frac12\pp{1-e^{-2(t_i-t_{i-1})\Lambda_\beta}}, i=1,\dots,n.
\eque
\end{lemma}
\begin{proof}
Given $\Lambda_\beta$, $\wt{R}_{\beta}$ is the collection of all points of a Poisson point process on $(0, 1)$ with intensity $\Lambda_\beta$. Then by independent scattering, we have that $\{B_i\}_{i=1, \dots, n}$ are conditionally independent since $\{(t_{i-1}, t_i]\}_{i=1, \dots, n}$ are disjoint. The corresponding parameter of each follows from the fact that, for a Poisson random variable $Z$ with parameter $\lambda>0$, $\proba(Z\ \rm odd) = (1-e^{-2\lambda})/2$. 
\end{proof}
Below is a summary of our improved algorithm for simulating $\vv D$.
\comment{\begin{algorithm} [H]\label{algo:2}
\caption{}
\begin{enumerate}
\item Generate the parameter $\lambda = G_1 G_{1-\beta}/G_{\beta}$ for $N_{\lambda}$.
\item Sample a random variable $U$ uniformly on $(0, 1)$.
\item From Lemma \ref{lem:B}, given $\lambda$, sample random variables $\{B_i\}_{i=1, \dots, n}$.
\item Compute $D_{t_i}= M_i \mmod 2$ for $i=1, \dots, n$ \eqref{eq:Mi}.
\end{enumerate}
\end{algorithm}}
\begin{algo}\label{algo:2}~
\begin{enumerate}
\item Sample $\Lambda_\beta\eqd G_1G_{1-\beta}/G_\beta$.
\item Given $\Lambda_\beta$, sample independent $B_i\sim {\rm Ber}(p_i(\Lambda_\beta)), i=1,\dots,n$ \eqref{eq:p_i}.
\item Sample $U\sim {\rm Unif}(0,1)$.
\item Compute $\vv M$ as in \eqref{eq:Mi} and $\vv D= \vv M \mmod 2$ as in \eqref{eq:DM}.
\end{enumerate}

\end{algo}

\begin{remark}\label{rem:complexity}
Algorithm \ref{algo:3} requires $Q_\beta$ number of exact locations of i.i.d.~random variables $\{U_i\}_{i \in \N}$, and this shall be repeated $N_{\alpha,\epsilon}$ times. The random variable $N_{\alpha,\epsilon}$ is Poisson and hence well concentrated at its mean $2^{1-\beta}\mu^\beta(E_0)C_\alpha\epsilon^{-\alpha}$. Viewing $N_{\alpha,\epsilon}$ as a fixed number for comparison, we see that this requires $\summ i1{N_{\alpha,\epsilon}}Q_{\beta,i}\cdot n$ number of iterations to sample the large-jump part, with $\{Q_{\beta,i}\}_{i\in\N}$ being i.i.d.~copies of $Q_\beta$. By the central limit theorem, we know that $N_{\alpha,\epsilon}^{-1/\beta}\summ i1{N_{\alpha,\epsilon}}Q_{\beta,i}$ has, for $\epsilon>0$ very small, approximately the totally skewed $\beta$-stable distribution (without finite $\beta$-th moment), say $Z_\beta$. So roughly Algorithm \ref{algo:3} has a complexity of order $Z_\beta \cdot N_{\alpha,\epsilon}^{1/\beta}\cdot n$. On the other hand, Algorithm \ref{algo:2} has a complexity of order $N_{\alpha,\epsilon}\cdot  n$, which is much lower.

\end{remark}

\subsubsection*{The small-jump part} In this case, simulating the small-jump part is straightforward, as the set-indexed fractional Brownian motion is $\{\B^{\beta/2}_\mu([0,t])\}_{t\ge 0} \equiv \{\B^{\beta/2}(t)\}_{t\ge0}$ the fractional Brownian motion with Hurst index $\beta/2\in(0,1/2)$ with covariance function as in \eqref{eq:cov_fBm}.
It is well known that fractional Brownian motions can be simulated in an exact and efficient manner by the circulant embedding method (e.g.~\citep{wood94simulation,dietrich97fast,perrin02fast}).

\subsubsection*{Simulations} In Figure \ref{fig:odd}, we provide a few simulation results for the odd-occupancy vector. In Figure \ref{fig:KSP}, we provide a few simulation results for the Karlin stable processes. The simulations are over a  lattice $\{i/n\}_{i=0,\dots,n}$ with $n = 1000$. 
\begin{figure}
\showFigure{
     \includegraphics[width=.9\textwidth]{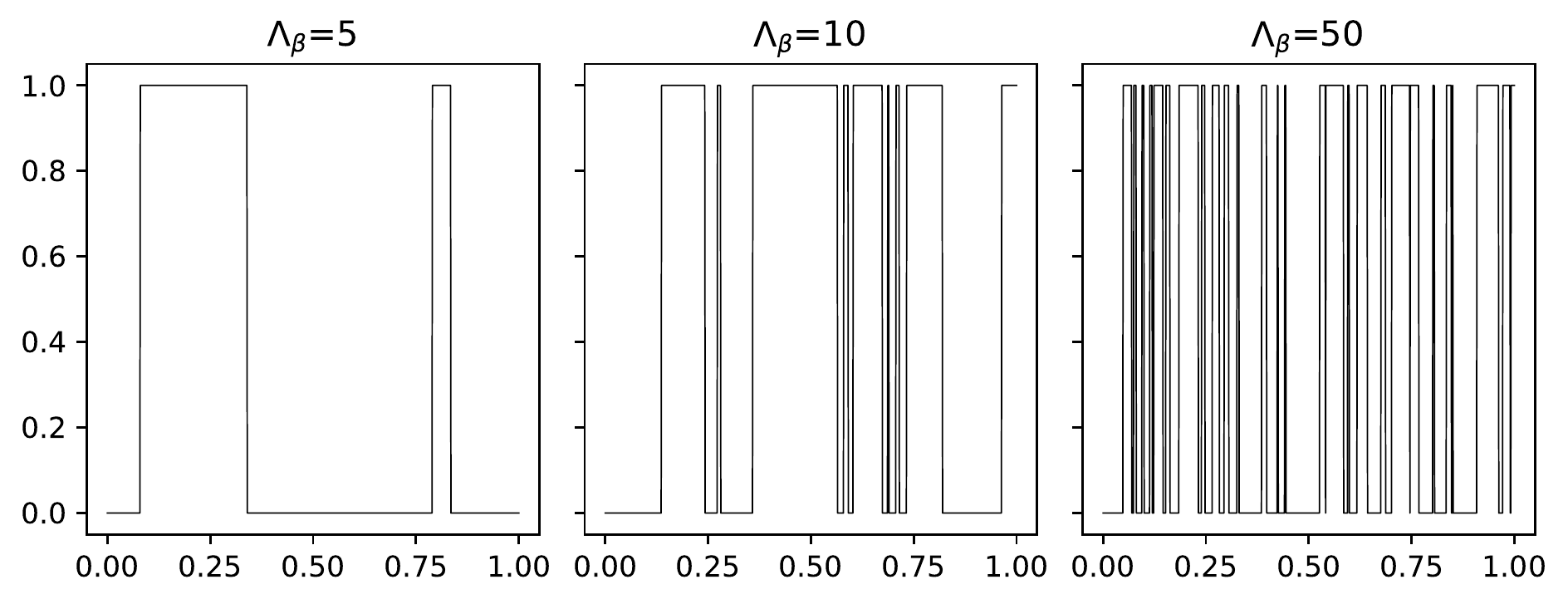}
     }
  \caption{Simulations of odd-occupancy vectors with different $\Lambda_\beta$. }
  \label{fig:odd}
\end{figure}

\begin{figure}[ht!]
\showFigure{
\begin{center}
  \includegraphics[width=.9\textwidth]{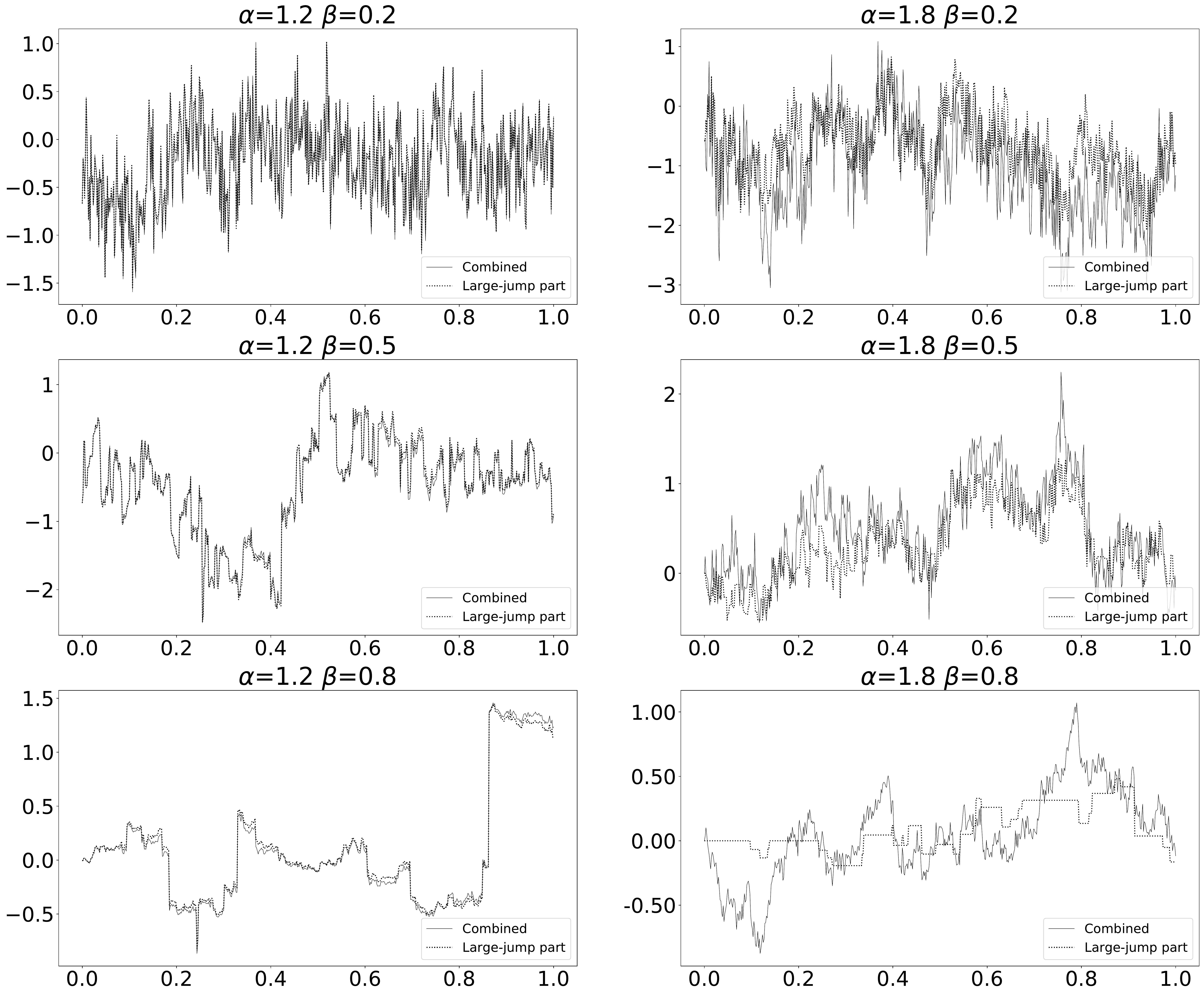}
  \end{center}
}
 \caption{Simulations of Karlin stable processes.}
\label{fig:KSP}
\end{figure}
\subsection{Multiparameter fractional stable fields}\label{sec:mfsf}
In this case, we take 
\equh\label{eq:mfsf_E0}
(E,\calE,\mu) = (\R_+^2,\calB(\R_+^2),{\rm Leb}), E_0 = [\vv0,\vv1], \mand \ccbb{A_\vvt} _{\vvt\in[\vv0,\vv1]} = \ccbb{[\vv0,\vvt]}_{\vvt\in[\vv0,\vv1]}.
\eque
(In this section, $[\vv a,\vv b] = [a_1,b_1]\times[a_2,b_2]$ for $\vv a = (a_1,a_2),\vv b=(b_1,b_2)\in\R_+^2$.)

\subsubsection*{The large-jump part} 
Again, $\{A_\vvt\}_{\vvt\in[\vv0,\vv1]^2}$ has a nice structure that we can exploit to obtain an efficient algorithm for sampling $\vv D$ as in Algorithm \ref{algo:2}. We only present a brief summary below as the proof is the same. This time the index lattice $T$ is given by 
\[
T := \ccbb{\pp{t_i\topp 1,t_j\topp 2}: t_i\topp r\in T\topp r, r=1,2}, \mwith T\topp r: = \ccbb{t\topp r_ i}_{i=1,\dots,n}\subset\R_+, r=1,2.
\]
Again we assume $t\topp r_i$ is increasing in $i$ for $r=1,2$.
This time we want to sample in law the vector $\vv D = \{D_{i,j}\}_{i,j=1,\dots,n}$ with 
\[
D_{i,j} \equiv D_{t_{i}\topp 1, t_{j}\topp 2}: = \inddd{\abs{R_{\beta} \cap [0, t_i\topp1]\times[0, t_{j}\topp2]}\ \rm odd}, i, j=1, \dots, n.
\]
Let $\Lambda_\beta$ be as before (see \eqref{sibuya1}). 
This time introduce
$\{B_{i,j}\}_{i,j=1,\dots,n}$ as conditionally independent Bernoulli random variables, given $\Lambda_\beta$, with parameters 
\equh\label{eq:pij}
p_{i,j}(\Lambda_\beta)=\frac12\pp{1-e^{-2(t_{i}\topp 1-t_{i-1}\topp1)(t_{j}\topp 2-t_{j-1}\topp 2)\Lambda_\beta}}, i,j=1,\dots,n,
\eque
with the convention $t_0\topp r = 0, r=1,2$. Let $\vv U$ be another independent uniform random vector in $(\vv0,\vv1)$. 
Then, with
\equh\label{eq:Mij}
M_{i, j} := \sum_{k=1}^{i}\sum_{\ell=1}^{j}B_{k, \ell} + \inddd{\vv U \in (0, t_i\topp1]\times (0, t_j\topp 2]}, i,j=1, \dots, n,
\eque
by the same argument as in Lemma \ref{lem:B} we have that
\equh \label{eq:DM1}
\ccbb{D_{i, j}}_{i, j=1, \dots, n} \eqd \ccbb{M_{i, j} \mmod 2}_{i,j=1,\dots,n}.
\eque
In summary, we use the following algorithm to sample the odd-occupancy vector $\vv D$ of the multiparameter fractional stable fields. 
\begin{algo}~
\begin{enumerate}
\item Sample $\Lambda_\beta \eqd G_1 G_{1-\beta}/G_{\beta}$.
\item Given $\Lambda_\beta$, sample independent $B_{i, j}\sim {\rm Ber}(p_{i,j}(\Lambda_\beta)), i,j=1,\dots,n$ \eqref{eq:pij}.
\item Sample $\vv U\sim {\rm Unif}(\vv0,\vv1)$.
\item Compute $\vv M$ as in \eqref{eq:Mij} and $\vv D = \vv M$ as in \eqref{eq:DM1}.
\end{enumerate}
\end{algo}
\subsubsection*{The small-jump part}

It turns out that the set-indexed process $\{\B^{\beta/2}_\mu([\vv0,\vvt])\}_{\vvt \in \R_{+}^2}\equiv \{\B^{\beta/2}(\vvt)\}_{\vv t\ge 0}$ becomes the multiparameter fractional Brownian motion \citep{herbin07multiparameter} with covariance function  \eqref{eq:cov_mfBm}.
This random field does not have stationary increments, and we are not aware of any exact sampling method that works efficiently with this covariance function. Instead, we propose to apply the following aggregation approximation for simulating the small-jump part. The general idea of aggregation approximation is, instead of applying the deterministic Cholesky decomposition of the given covariance matrix $\Sigma$, to find an easy-to-simulate random vector ($\vv D$ here) so that $\Sigma = \esp (\vv D'^t\vv D)$ (here $\vv D'^t$ is the transpose of $\vv D'$, an independent copy of $\vv D$). Below, recall that in this section we identify $\calA_0 = \{A_\vvt\}_{\vvt\in[\vv0,\vv1]}$. We also keep the factor $\mu(E_0)$ below, although for set-indexed fractional Brownian motion \eqref{eq:mfsf_E0}, $\mu(E_0) = 1$. 
\begin{proposition} \label{prop:small2}
Let $\{\varepsilon_j\}_{j \in \N}$ be a sequence of i.i.d.~standard normal random variables and $\{\vv D_{j}\}_{j \in \N}$ be i.i.d.~copies as in \eqref{eq:large-jump}. Then we have 
\equh \label{multiapprox}
(2^{1-\beta}\mu^\beta(E_0))^{1/2}\cdot\ccbb{\frac{1}{\sqrt m}\sum_{j=1}^m \varepsilon_j D_{j, \vvt}}_{\vvt \in [\vv0,\vv1]} \fddto\ccbb{\B^{\beta/2}(\vvt)}_{\vvt \in [\vv0,\vv1]},
\eque
as $m\to\infty$, with $\B^{\beta/2}$ determined by \eqref{eq:set_indexed}.
\end{proposition}
\begin{proof}
By the multivariate central limit theorem, it suffices to compute to the asymptotic covariance of the left hand side of \eqref{multiapprox}. 
That is, for $\vv s,\vv t \in [\vv0,\vv1]$, 
\[
 \cov (D_{ \vvs}, D_{ \vvt}) =  \E(D_{\vvs}D_{\vvt})  =  \P\pp{\abs{R_{\beta} \cap A_{\vvs}}\ {\rm odd}, \abs{R_{\beta} \cap A_{\vvt}}\ {\rm odd}}.
 \]
We have seen this computation in \eqref{eq:A1A2}. 
\end{proof}
Since $|D_\vvt|\le 1$, we have a Berry--Esseen upper bound as $3.3/\sqrt m$ \citep[Theorem 3.4]{chen11normal}. Applying the standard Berry--Esseen bound for the sum of i.i.d.~centered random variables with unit variance \citep{korolev12improvement}, we have  (recall \eqref{eq:ED_t})
\begin{align}
C_{\rm BE}\frac{\esp|D_\vvt|^3}{(\esp|D_\vvt|^2)^{3/2}}m^{-1/2} & = C_{\rm BE}\proba(|R_\beta\cap A_\vvt|\ {\rm odd})^{-1/2}m^{-1/2}\nonumber\\
& = C_{\rm BE}\pp{2^{\beta-1}\frac{\mu^\beta(A_t)}{\mu^\beta(E_0)}}^{-1/2}m^{-1/2},\label{eq:m}
\end{align} as a Berry--Esseen upper bound for the convergence of \eqref{multiapprox}. 
\subsubsection*{Simulations}
\begin{figure}[ht!]
\showFigure{
\begin{center}
   \includegraphics[width=.9\textwidth]{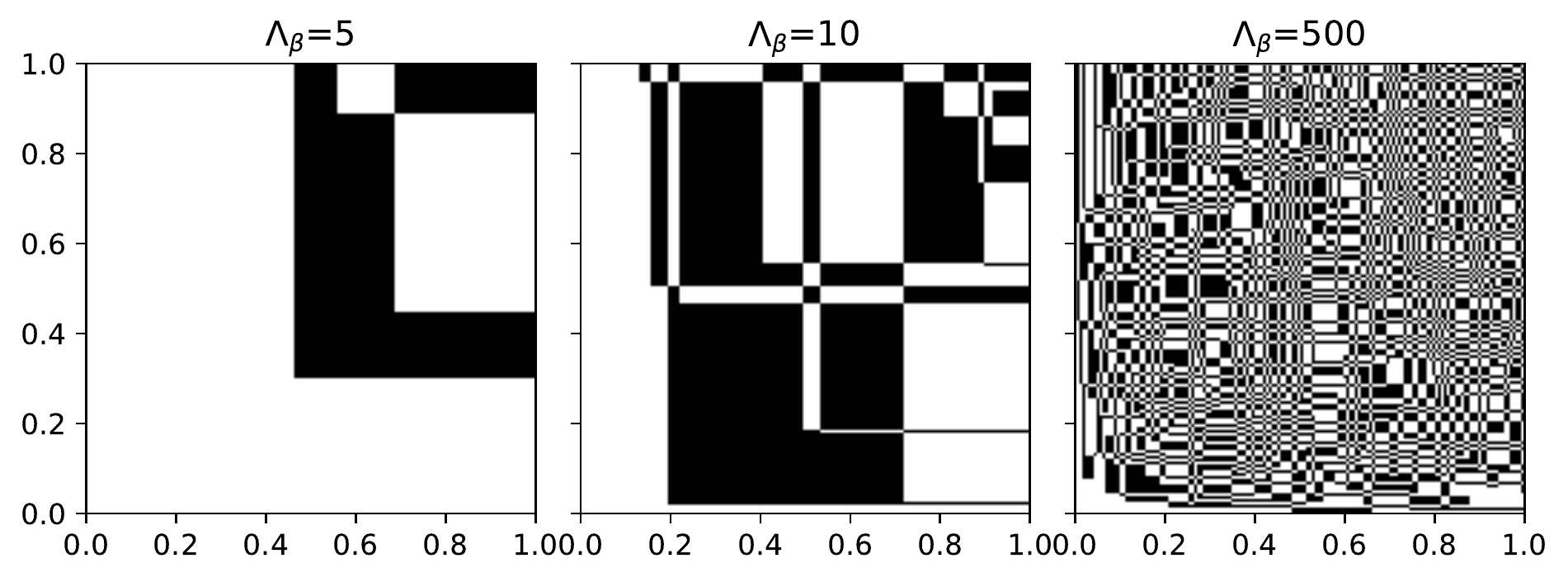}
\end{center}
}
 \caption{Simulations for odd-occupancy vectors for multiparameter stable fields with different values of $\Lambda_\beta$.}
 \label{fig:odd_multi}
\end{figure}
\begin{figure}[ht!]
\showFigure{
\begin{center}
 \includegraphics[width=.9\textwidth]{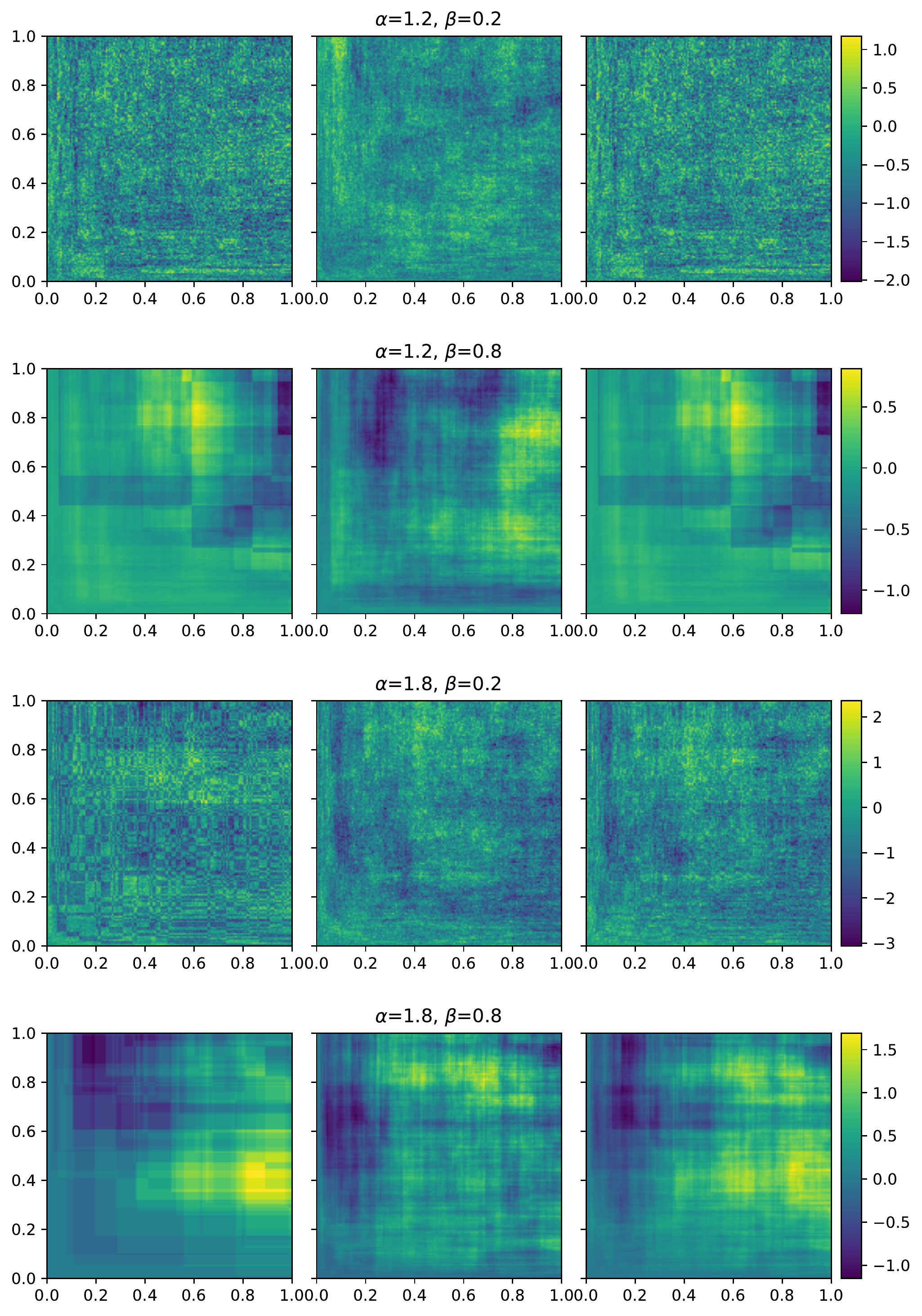}
\end{center}
}
 \caption{Simulations for multiparameter fractional stable fields. From left to right: the large-jump parts, the small-jump parts, and the combined fields.}
 \label{fig:multiparameter}
\end{figure}
Figure \ref{fig:odd_multi} provides a few simulations of the odd-occupancy vectors. 
Figure \ref{fig:multiparameter} provides a few simulations for the multiparameter fractional stable fields. The random field is sampled over a $300\times 300$ lattice. For the small-jump part we take $m=2500$ in Proposition \ref{prop:small2} in view of the Berry--Esseen bound \eqref{eq:m} (so that $m^{-1/2} = 2\%$). 
\subsection{Fractional L\'evy--Chentsov stable fields}\label{sec:fLC}
In this case, we take
\[
(E,\calE,\mu) = \pp{\S^1\times \R_+,\calB(\S^1\times\R_+),d\vv sdr},
\]
where $d\vv sdr$ is the product measure of the uniform measure $d\vv s$ on $\S^1$ and the Lebesgue measure $dr$ on $\R_+$, 
and in practice we may restrict to
\[
E_0 = \S^1\times[0,\sqrt 2] \qmand \{A_\vvt\}_{\vvt\in[\vv0,\vv1]}= \ccbb{(\vvs, r): \vvs \in \S^1, 0 < r < \aa{\vvs, \vvt}}_{\vvt\in[\vv0,\vv1]},
\]
with $\mu(E_0) = \sqrt 2 \cdot 2\pi$. (Actually, one could further restrict to $([0,\pi]\cup[3\pi/2,2\pi))\times[0,1]\subset E_0$ to gain some extra computational efficiency.)
In this case, $\{\B^{\beta/2}_\mu(A_\vvt)\}_{\vvt \in[\vv0,\vv1]}\equiv\{\B^{\beta/2}(\vvt)\}_{\vvt\in\in[\vv0,\vv1]}$ becomes a fractional L\'evy Brownian field, a centered Gaussian random field with covariance function \eqref{eq:cov_fLBf}.

\subsubsection*{The large-jump part} The nice lattice structure of $\{A_\vvt\}$ in the previous two examples is lost here, and it seems that we have to reply on Algorithm \ref{algo:3} to sample the large-jump part, which is computationally inefficient. 
\subsubsection*{The small-jump part}
It is well known that the intrinsic embedding method by \citet{stein02fast} can be applied to simulate exactly and efficiently the fractional L\'evy Brownian fields. 
\subsubsection*{Simulations}
\begin{figure}[ht!]
\showFigure{
\begin{center}
 \includegraphics[width=.9\textwidth]{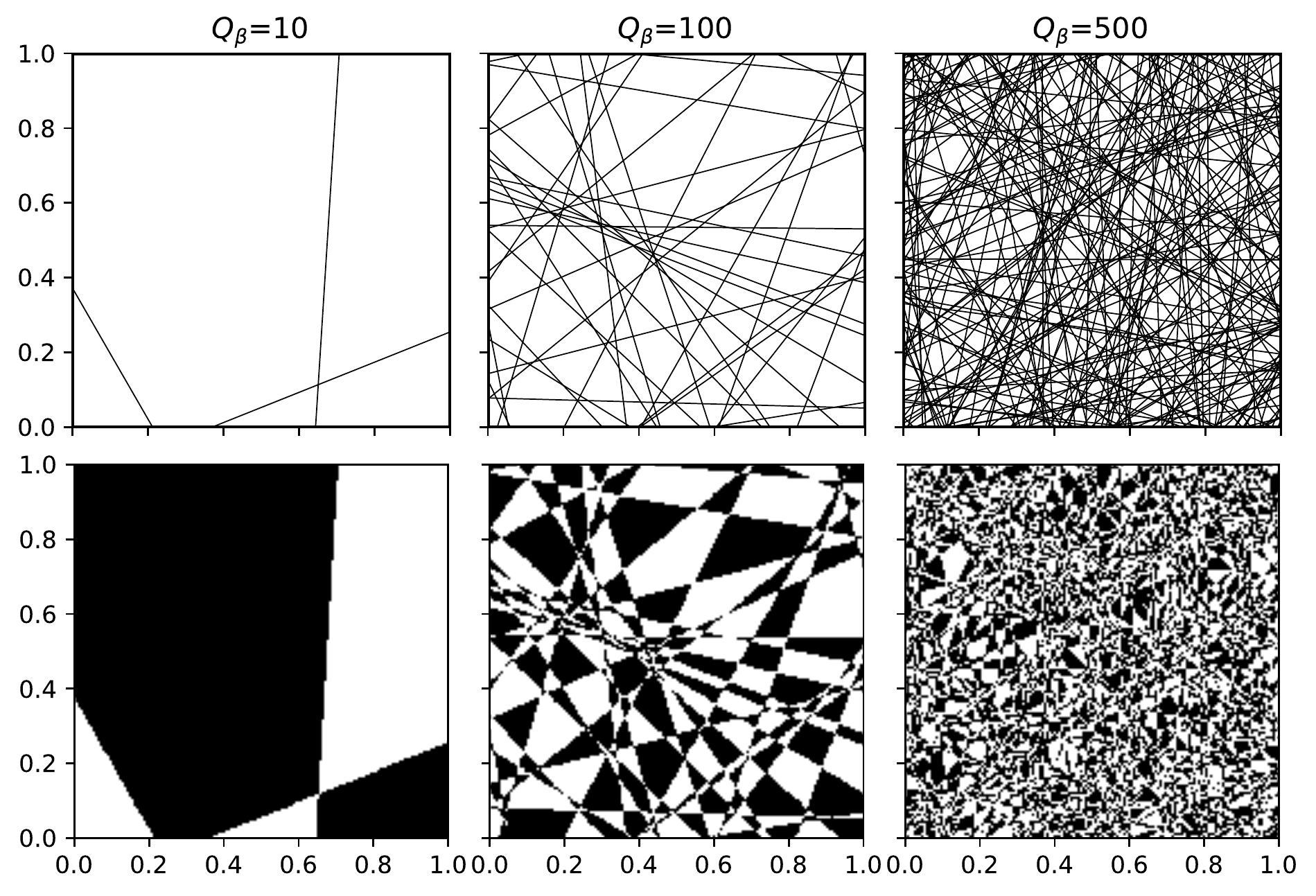}
   \end{center}
}
 \caption{Simulations for odd-occupancy vectors for fractional L\'evy--Chentsov stable fields with different $Q_\beta$. The plots in first row are i.i.d.~$Q_\beta$ hyperplanes (some may not intersect the region $[0,1]^2$), and the plots in the second row are the corresponding odd-occupancy vectors over a $300\times300$ lattice.}
 \label{fig:odd_fLC}
\end{figure}
\begin{figure}[ht!]
\showFigure{
\begin{center}
 \includegraphics[width=.9\textwidth]{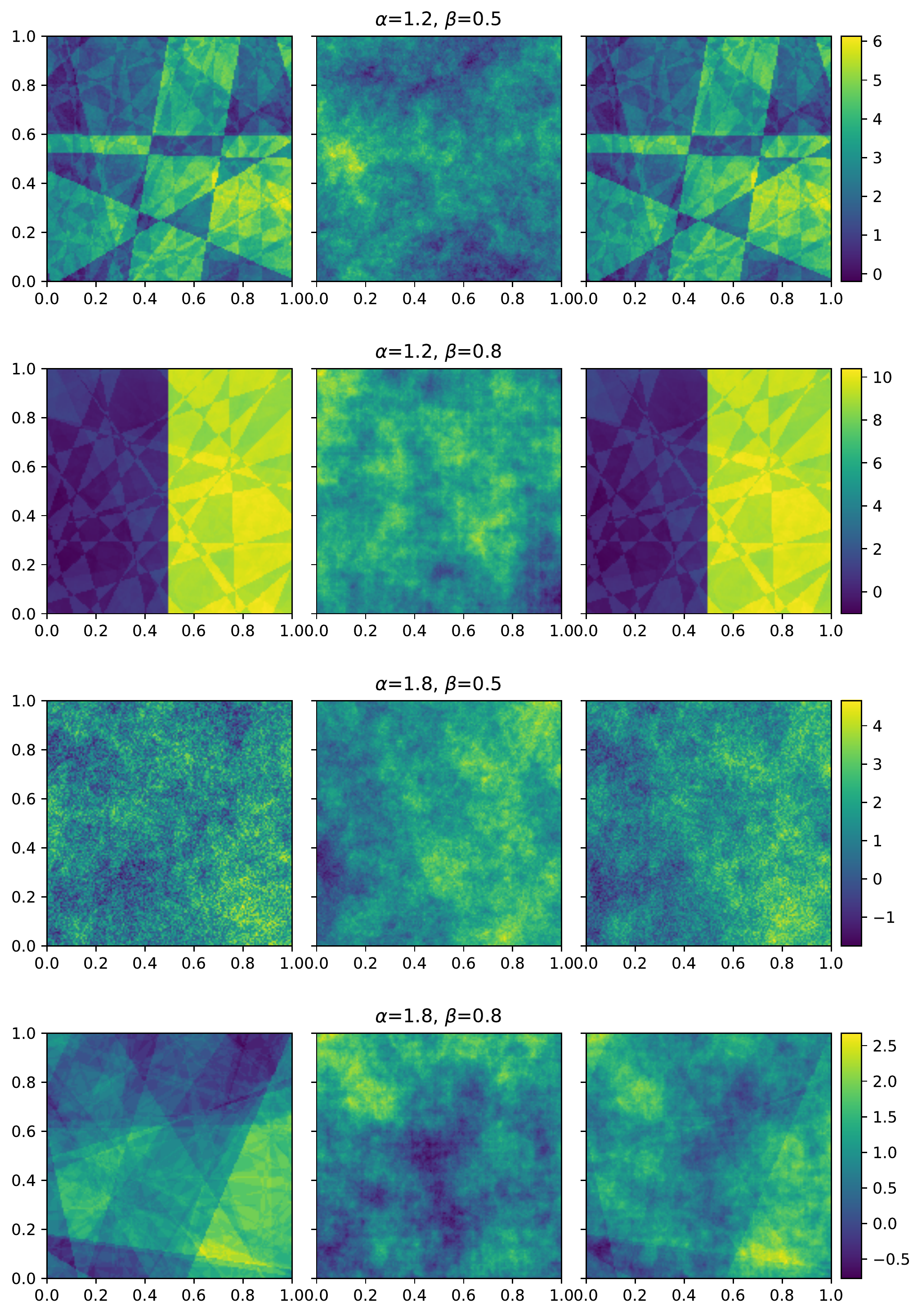}
\end{center}
}
 \caption{Simulations for fractional L\'evy--Chentsov stable fields.  From left to right: the large-jump parts, the small-jump parts, and the combined fields.}
 \label{fig:fLC}
\end{figure}
Figure \ref{fig:odd_fLC} provides a few simulations for the odd-occupancy vectors for the fractional L\'evy--Chentsov stable fields. 
Figure \ref{fig:fLC} provides a few simulations  for the fractional L\'evy--Chentsov stable fields. The random fields are sampled over a $300\times 300$ lattice.

\subsection{Spherical fractional L\'evy--Chenstov stable fields}\label{sec:sfLC}
In this case, we take
\[
(E,\calE,\mu) = (\S^2,\calB(\S^2),d \vvs), E_0 = E,
\]
where $d\vvs$ is the Lebesgue measure on the unit sphere $\S^2$ in $\R^3$, and
\[
A_{\vv x}=H_{\vv x} \triangle H_{\vv o}, \vv x \in \S^2 \qmwith H_{\vv x}: = \ccbb{\vv y \in \S^2: \langle \vv x, \vv y\rangle > 0},
\]
where $\vv o\in \S^2$ is the fixed north pole, and $H_{\vv x}$ is the hemisphere of $\S^2$ determined by $\vv x$. 
The spherical fractional L\'evy--Chentsov stable field, denoted by $\{Y_{\alpha,\beta}(\vv x)\}_{\vv x \in\S^2}\equiv \{Y_{\alpha,\beta}(A_{\vvx})\}_{\vvt\in\S^2}$, can be obtained by
\equh\label{eq:pin_down}
\ccbb{Y_{\alpha,\beta}(\vvx)}_{\vvx\in\S^2} \eqd\ccbb{\wt Y_{\alpha,\beta}(\vvx) - \wt Y_{\alpha,\beta}(\vv o)}_{\vvx\in\S^2}, \qmwith \wt Y_{\alpha,\beta}(\vvx) := \wt Y_{\alpha,\beta}(H_\vvx), \vvx\in\S^2.
\eque
The random field $\{\wt Y_{\alpha,\beta}(\vvx)\}_{\vvx\in\S^2}$ is again a special case of Karlin random fields. In addition, it is rotationally stationary (a.k.a.~strongly isotropic), and the discussions below are for  $\wt Y_{\alpha,\beta}$ instead of $Y_{\alpha,\beta}$. 

 \subsubsection*{The large-jump part} We rely on Algorithm \ref{algo:3} to simulate the large-jump part. 
 \subsubsection*{The small-jump part}
 An advantage of working with $\wt Y_{\alpha,\beta}$ instead of $Y_{\alpha,\beta}$ is that now, Proposition \ref{prop:small1} says that the small-jump part is approximated by a rotationally stationary spherical Gaussian field, denoted by  $\{\wt \B^{\beta/2}(\vv x)\}_{\vvx\in\S^2}$. Thanks to the rotational stationarity, such Gaussian random fields  can be simulated fast and exactly by the circulant embedding method \citep{cuevas20fast}. 

It remains to compute the covariance explicitly. 
 In view of Proposition \ref{prop:small1}, $\wt \B^{\beta/2}$ is a set-indexed fractional Brownian motion with the same law as $Y_{2,2H}(H_{\vv x})$ (see \eqref{eq:integral}), where $H_{\vv x}$ is the hemisphere determined by $\vv x\in \S^2$ and $\mu$ the Lebesgue measure on $\S^2$ so that $\mu(H_{\vv x}) = 2\pi$ and $\mu(H_{\vv x}\Delta H_{\vv y}) = 4\dd(\vv x,\vv y)$. Therefore,  we have
\begin{align*}
\cov\pp{\wt \B^{\beta/2}(\vv x),\wt \B^{\beta/2}(\vv y)} &= \cov(Y_{2,2H}(H_{\vv x}),Y_{2,2H}(H_{\vv y})) \\
&= \frac12\pp{\mu^{2H}(H_{\vv x}) +\mu^{2H}(H_{\vv y}) - \mu^{2H}(H_{\vv x}\Delta H_{\vv y})}\\
& = (2\pi)^{2H}\pp{1-\frac12\pp{\frac2\pi}^{2H}\dd^{2H} (\vv x,\vv y)}.
\end{align*} 
\begin{figure}[ht!]
\showFigure{
\begin{center}
 \includegraphics[width=.9\textwidth]{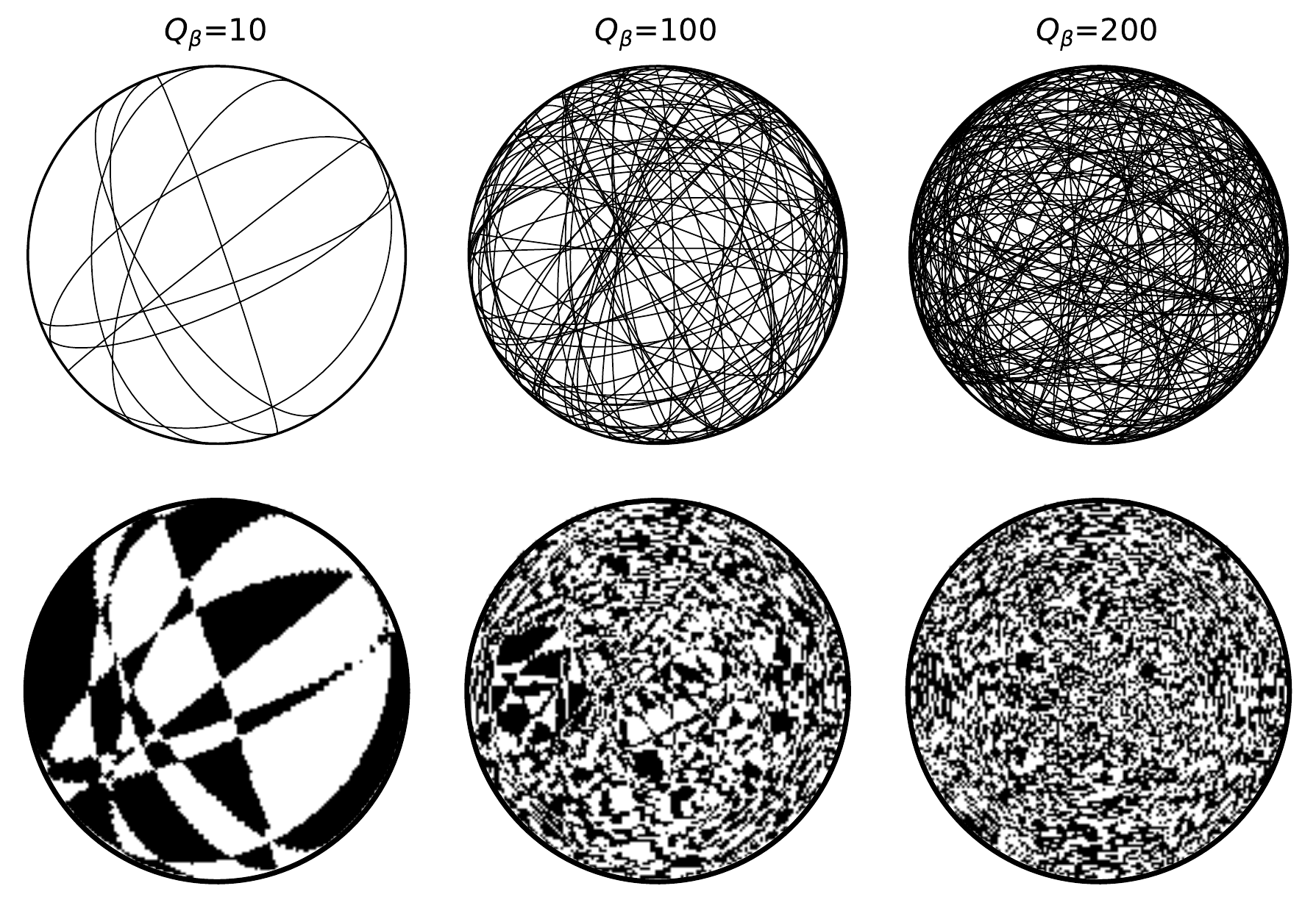}
   \end{center}
}
 \caption{Simulations for odd-occupancy vectors for spherical fractional L\'evy--Chentsov stable fields for different $Q_\beta$. 
 The plots in first row are the great circles corresponding to i.i.d.~$Q_\beta$ points from the sphere, and the plots in the second row are the corresponding odd-occupancy vectors over a $300\times150$ lattice in polar coordinates. }
 \label{fig:odd_sfLC}
\end{figure}
\begin{figure}[ht!]
\showFigure{
\begin{center}
 \includegraphics[width=.9\textwidth]{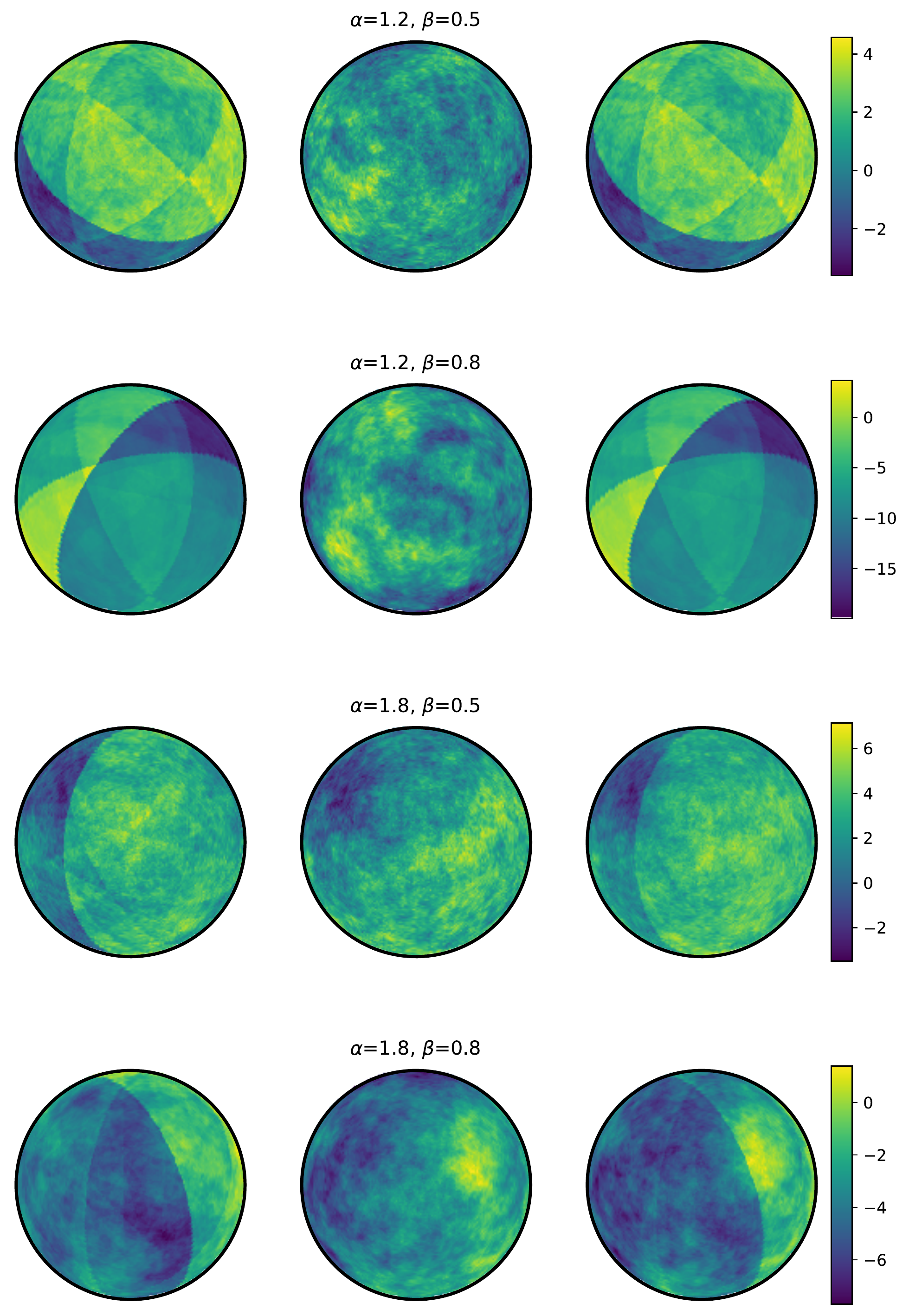}
\end{center}
}
 \caption{Simulations for rotationally stationary spherical fractional L\'evy--Chentsov stable fields.  From left to right: the large-jump parts, the small-jump parts, and the combined fields.
 }
 \label{fig:sfLC}
\end{figure}
\subsubsection*{Simulations}
Figure \ref{fig:odd_sfLC} provides a few simulations for the odd-occupancy vectors for spherical fractional L\'evy--Chentsov fields. 
Figure \ref{fig:sfLC} provides a few simulations for the  spherical fractional L\'evy--Chentsov fields. The spherical random fields are sampled over a $300\times 150$ lattice in the polar coordinates. 
For simulation examples of $Y_{\alpha,\beta}$, see Figure \ref{fig:samples}, where we sampled the approximated $\wt Y_{\alpha,\beta}$ first and applied the pinning-down relation \eqref{eq:pin_down}.

\subsection*{Acknowledgement}
The authors would like to thank an anonymous referee for very helpful and constructive comments.  YW would like to thank Yimin Xiao for helpful discussions on spherical fractional Brownian motions. ZF and YW's research were partially supported by Army Research Office grant W911NF-17-1-0006.

\appendix
\section{A general framework for simulating stable processes}\label{sec:general}
The framework here can be read from \citep{cohen08general} where an essentially more general one for infinitely-divisible processes is explained in details. We only focus on a subclass of $S\alpha$S processes, of which the task is significantly simplified (see Remark \ref{rem:Lacaux}).  Namely, for some measurable space $(S,\calS)$ equipped with a {\em finite measure} $m$ and a family of {\em square integrable functions} $\{f_t\}_{t\in T}$ on $(S,m)$, we are interested in simulating S$\alpha$S processes defined as
\equh \label{seriesrep}
X(t):=\sum_{j \in \N} \eta_{\alpha,j} f_t(W_j), t\in T, \alpha \in (0, 2), 
\eque
where 
\[
\ccbb{\pp{\eta_{\alpha,j},W_j}}_{j\in\N} \sim \PPP\pp{\frac{\alpha C_\alpha}2|y|^{-\alpha-1}dydm}.
\]
\begin{remark}
Alternatively, the above can be viewed as a Poisson point process with i.i.d.~marks, with $\{\eta_{\alpha,j}\}_{j\in\N}\sim \PPP((1/2)C_\alpha m(S) \alpha |y|^{-\alpha-1}dy)$ on $\wb\R\setminus\{0\}$ and $\{W_j\}_{j \in \N}$ as i.i.d.~random elements in $S$ with law $m(\cdot)/m(S)$, two families being independent. This representation is helpful for some analysis of the stable processes, but is not needed in our proofs. 
\end{remark}
The definition \eqref{seriesrep} has the following stochastic-integral representation
\equh \label{integralrep}
\ccbb{X(t)}_{t \in T} \eqd \ccbb{\int_{S} f_{t}(s) M_{\alpha}(ds)}_{t \in T}, \alpha \in (0, 2), 
\eque
where $M_{\alpha}$ is an S$\alpha$S random measure on $(S, \calS)$ with control measure $m$ \citep[Corollary 3.10.4]{samorodnitsky94stable}.  
In general, the representations of stable processes, in particular the choices of $(S,m)$, are not unique, and a good choice may increase significantly the efficiency of simulation method. 

It is well known that, when $\alpha \in (0, 2)$, there are no exact simulation methods for most S$\alpha$S processes. In the seminal work of \citet{asmussen01approximations}, it was pointed out that in simulations, the S$\alpha$S process should be decomposed into the large-jump and small-jump parts, and then the two parts could be simulated independently. Namely, let $\epsilon > 0$, in view of \eqref{seriesrep}, the process $\{X(t)\}_{t \in T}$ can be written as the sum of two independent processes
\[
X(t) = X_{\epsilon, 1}(t) + X_{\epsilon, 2}(t),
\]
with $X_{\epsilon,1}$ and $X_{\epsilon,2}$ given by
\[
X_{\epsilon,1}(t) := \sif n1 \eta_{\alpha,n}f_t(W_n)\inddd{\eta_{\alpha,n}\ge \epsilon},\mand 
X_{\epsilon,2}(t) := \sif n1 \eta_{\alpha,n}f_t(W_n)\inddd{\eta_{\alpha,n}< \epsilon}.
\]

The two processes are referred as the {\em large-jump} and the {\em small-jump} parts, respectively from now on.
For the large-jump part,  thanks to our assumption that $m$ is finite on $(S,\calS)$, it is immediately seen that $X_{\epsilon,1}$ has a compound-Poisson representation as
\equh\label{eq:compound}
\ccbb{X_{\epsilon,1}}_{t\in T} \eqd \ccbb{\summ j1{N_{\alpha,\epsilon}}V_{\alpha,\epsilon,j}f_t(W_j)}_{t\in T},
\eque
where $N_\epsilon$ is a Poisson random variable with parameter $C_\alpha m(S) \epsilon^{-\alpha}$, $W_j$ are as before, $V_{\alpha,\epsilon,j}$ has probability density $(1/2)\epsilon^\alpha\alpha |y|^{-\alpha-1}, |y|>\epsilon$, and all random variables are independent. An exact simulation of $X_{\epsilon,1}$ in view of \eqref{eq:compound} is straightforward.

The small-jump part $\{X_{\epsilon, 2}(t)\}_{t \in T}$  is an infinitely-divisible process that can be approximated by a Gaussian process, as summarized in the following proposition. The proof is essentially the same as  \citet[Theorem 2.1]{asmussen01approximations}; see also \citep[Lemma 4.1]{lacaux04series} and \citep[Proposition 5.1]{cohen08general}. For the sake of completeness we include a proof here again. 
Let $\nu_\alpha(dx)$ denote the L\'evy measure   for standard S$\alpha$S distribution \[
\nu_\alpha(dx) := \frac{\alpha C_\alpha}2 |x|^{-1-\alpha}dx, x\ne 0.
\]
Introduce 
\[
\sigma_\alpha(\epsilon) := \pp{\int_{-\epsilon}^{\epsilon}v^2 \nu_\alpha(dv)}^{1/2} = \pp{\alpha C_\alpha \int_{0}^{\epsilon}v^{1-\alpha} dv}^{1/2} = \pp{ \frac{\alpha C_\alpha}{2-\alpha}}^{1/2}\epsilon^{1-\alpha/2}.
\]
\begin{proposition}\label{prop:small}
Assume that $f_t\in L^2(S,m)$ for all $t\in T$. Then 
\[
\ccbb{\frac{X_{\epsilon, 2}(t)}{\sigma_\alpha(\epsilon)}}_{t \in T} \stackrel{f.d.d.}\weakto \ccbb{\mathbb G(t)}_{t \in T},
\]
as $\epsilon\downarrow0$,
where $\{\G(t)\}_{t \in S}$ is a centered Gaussian process with  covariance function
\[
\cov (\G(t_1), \G(t_2)) = \int_{S} f_{t_1}(s)f_{t_2}(s) m(ds), t_1, t_2 \in T.
\]
\end{proposition}
The tightness of the sequence $\{X_{\varepsilon,2}\}_{\epsilon>0}$ was also established in a few earlier investigated cases \citep{asmussen01approximations,lacaux04series}. 
Note that the Gaussian process $\G$ that arises in the limit shares the same form of integral representations as the original S$\alpha$S process $X$, with the S$\alpha$S random measure replaced by a Gaussian random measure ($\alpha = 2$).

\begin{remark}\label{rem:Lacaux}
Most examples of interest in  \citep{lacaux04series,lacaux04real,cohen08general} are such that $S = \Rd$ equipped with the control measure $m$ being the Lebesgue measure. Then, the large-jump part does not have compound--Poisson representation; it is known as a shot-noise model over $\Rd$ in the literature \citep{vervaat79stochastic}. Simulating of shot-noise models requires another approximation,  with key ideas from \citep{rosinski01series}. On the other hand, the treatment for approximation the small-jump part remains the same for different choices of $(S,m)$. 
From this point of view, working with a generic $(S,m)$ instead of $(\Rd,{\rm Leb})$ as in earlier references does not bring new technical challenges in analysis immediately: choosing $m$ to be finite on $S$ even simplifies our task. 

It is worth noting that the assumption on the finiteness on $m$ is not essential, as one could also apply a change-of-measure trick to work with a different representation satisfying this property. The essential constraint here is the $L^2$-integrability of the functions $f_t$ (after change of measure) that is needed for the Gaussian approximations of the small-jump part (for \eqref{seriesrep} to be a well defined S$\alpha$S process it suffices to have $f_t\in L^\alpha$ in general). 
Another notable example of S$\alpha$S processes that fits into the framework presented here is the one recently introduced in \citep{owada15functional}, where $S$ takes a more abstract space than $\Rd$.
\end{remark}

\begin{proof}[Proof of Proposition \ref{prop:small}]
We start by providing some background on infinitely-divisble processes. 
As an infinitely-divisible process, by \citep[Theorems 3.3.2 and 3.4.3]{samorodnitsky16stochastic}, \eqref{integralrep} also can be written as the following integral representation
\equh\label{eq:stochastic_integral}
\ccbb{X(t)}_{t \in T} \eqd \ccbb{\int_{S} f_t(s)M^{id}_{\alpha}(ds)}_{t \in T},
\eque
where $M^{id}_{\alpha}$ is an infinitely-divisible random measure on $S$ with control measure $dm$, and $M^{id}_\alpha$ is uniquely determined by {\em local characteristics} $\sigma^2 \equiv 0, b\equiv 0, \rho(s,\cdot) = \nu_\alpha(\cdot)$ \citep[P.86]{samorodnitsky16stochastic}.
(The infinitely-divisible random variable $X(t)$ has L\'evy measure on $\R$ as the push-forward measure 
\equh\label{eq:Levy}
\mu_{f_t} := (m\times \nu_\alpha)\circ T_{f_t}\inv \qmwith T_{f_t}(s,x) := xf_t(s), s\in S,x\in\R,
\eque
see \citep[Theorem 3.3.2]{samorodnitsky16stochastic}, although we do not gain anything in this proof by using $\mu_{f_t}$.)

We shall understand stochastic-integral representations as in \eqref{eq:stochastic_integral} via their corresponding characteristic functions of finite-dimensional distributions based on local characteristics, 
namely with $\summ j1d \theta_jf_{t_j}(s) \equiv g(s)$, 
\equh\label{eq:LK}
\esp \exp\pp{i\summ j1d \theta_jX_{t_j}}  = \exp\pp{\int_S\int_\R \pp{e^{i g(s)x}-1-ig(s)\dbb x}\nu_\alpha(dx)m(ds)},
\eque
where
\[
\dbb x  = \begin{cases}
x & |x|\le 1,\\
-1 & x<1,\\
1 & x\ge 1.
\end{cases}
\]
Then, $X_{\epsilon,2}$ has the similar integral representation as \eqref{eq:stochastic_integral} with $M^{id}_\alpha$ modified by replacing the L\'evy measure $\nu_\alpha$ by the truncated measure $\inddd{|v| < \epsilon}\nu_\alpha(dv)$. 
Now we consider for $d \in \N$, $\vvt=(t_1, \dots, t_d) \in T^d$ and $\vv\theta=(\theta_1, \dots, \theta_d) \in \R^d$,
\[
g_{\vv\theta,\vv t}(s): = \sum_{j=1}^{d} \theta_j f_{t_j}(s).
\]
Then the characteristic function of finite-dimensional distribution of $X_{\epsilon, 2}(t)$ is given by (thanks to the symmetry of $\nu_{\alpha}$, replacying $g(s)\dbb x$ in \eqref{eq:LK} by $g(s)x\inddd{|g(s)x|\le 1}$)
\[
 \E \exp\pp{i \frac{\summ j1d\theta_jX_{\epsilon,2}(t_j)}{\sigma_\alpha(\epsilon)}}  = \exp\pp{\int_S I_{\alpha,\epsilon}(g_{\vv\theta,\vv t}(s))m(ds)},
 \]
with
\begin{align*}
I_{\alpha,\epsilon}(y) & := \int_{\R} \pp{\exp\pp{i\frac{yx}{\sigma_\alpha(\epsilon)}} - 1 - i\frac{y}{\sigma_\alpha(\epsilon)} \dbb x}\inddd{|x|\le\epsilon} \nu_{\alpha}(dx)\\
& = \int_{-\epsilon}^\epsilon\pp{ \exp\pp{i\frac{yx}{\sigma_\alpha(\epsilon)}} - 1 - i\frac{yx}{\sigma_\alpha(\epsilon)} }\nu_{\alpha}(dx),
\end{align*}
where we dropped $\dbb x$ on the right-hand side of first line above thanks to the symmetry of $\nu_\alpha$. 
Now, since $\sigma_\alpha(\epsilon)/\epsilon\to\infty$ as $\epsilon\downarrow 0$, we have
\[
I_{\alpha,\epsilon}(y)  \sim \int_{-\epsilon}^\epsilon -\frac{y^2x^2}{2\sigma_\alpha(\epsilon)^2}\nu_\alpha(dx) = -\frac{y^2}2 \frac{\int_{-\epsilon}^\epsilon x^2\nu_\alpha(dx)}{\sigma_\alpha(\epsilon)^2} = -\frac{y^2}2.
\]
In addition, for all $y \in \R$, 
$\abs{I_{\alpha,\epsilon}(y)} \leq y^2/2$ (since $|e^{ix}-1-ix|\le x^2/2$). Therefore by the dominate convergence theorem we have 
\begin{align*}
\lim_{\epsilon \downarrow 0}  \E \exp\pp{i \frac{\summ j1d\theta_jX_{\epsilon,2}(t_j)}{\sigma_\alpha(\epsilon)}}  & = \lim_{\epsilon \downarrow 0}  \E \exp\pp{\int_S I_{\alpha,\epsilon}(g_{\vv\theta,\vv t}(s))m(ds)}\\
& = \exp \pp{-\frac{1}{2}\int_{S} \abs{g_{\vv\theta,\vv t}(s)}^2 m(ds)}.
\end{align*}
Now, we read the right-hand side as the the characteristic function of $\summ j1d \theta_j\G(t_j)$, which completes the proof.
\end{proof}
So for the small-jump part, in practice we shall pick a small number $\epsilon>0$ and  apply the approximation
\[
\ccbb{X_{\epsilon,2} (t)}_{t \in T} \approx \ccbb{\sigma_\alpha(\epsilon) \mathbb G(t)}_{t \in T},
\] 
for the corresponding Gaussian process in Proposition \ref{prop:small}. The replacement of small-jump part by a Gaussian process is crucial in view of numerical analysis. For example, for L\'evy-driven stochastic differential equations, the performance of approximation schemes is much better with the Gaussian approximation than simply neglecting all the small jumps. See \citep{fournier11simulation} and references therein for a detailed investigation.
\begin{remark}\label{rem:BE}
As in earlier results, one could also have an Berry--Esseen bound on the pointwise approximation, thanks to \citep[Theorem 3.1]{asmussen01approximations}, \citep[Lemma 4.1]{lacaux04series}: letting 
\[
s^2_\epsilon(t) = \esp X_{\varepsilon,2}^2(t) = \int_S |f_t|^2dm \int_{-\epsilon}^\epsilon x^2\nu_\alpha(dx) = \var(\G(t))\sigma_\alpha^2(\epsilon),
\] 
we have immediately the following rate for the convergence in Proposition \ref{prop:small} (recall the L\'evy measure in \eqref{eq:Levy})
\begin{multline*}
\sup_{x\in\R}\abs{\proba\pp{\frac{X_{\varepsilon,2}}{\sigma_\alpha(\epsilon)}\le x} - \proba(\G(t)\le x)} \\
\le C_{\rm BE} \frac{\int_S|f_t|^3dm\int_{-\epsilon}^\epsilon |x|^3\nu_\alpha(dx)}{s_\epsilon^3(t)}
\le C_{\rm BE} \frac{\int_S|f_t|^3dm}{(\int_S|f_t|^2dm)^{3/2}}\frac{(2-\alpha)^{3/2}}{(3-\alpha)\sqrt{\alpha C_\alpha}}\epsilon^{\alpha/2}, 
\end{multline*}
where $C_{\rm BE}$ is the constant in standard Berry--Esseen upper bound for partial sum of centered i.i.d.~random variables with unit variance. The value $C_{\rm BE} = 0.7975$ was used in the aforementioned references, and this value has been improved to 0.4785 in \citep{korolev12improvement}. 
\end{remark}

\bibliographystyle{apalike}

\bibliography{references,references18}

\end{document}